\documentclass[twoside,a4paper,11pt]{amsart}
\pdfoutput=1
\usepackage{amsmath,amssymb,amsthm,amscd}
\usepackage[]{latexsym,amssymb,amsmath,amsfonts, amsthm, verbatim}
\usepackage[all,cmtip,ps]{xy}
\usepackage[latin1]{inputenc}
\usepackage[english]{babel}

\usepackage{graphicx}

\DeclareMathOperator{\add}{add}
\DeclareMathOperator{\Add}{Add}
\DeclareMathOperator{\Prod}{Prod}
\DeclareMathOperator{\Gen}{Gen}
\DeclareMathOperator{\Cogen}{Cogen}
\DeclareMathOperator{\vect}{vect}

\DeclareMathOperator{\coh}{coh}
\DeclareMathOperator{\Qcoh}{Qcoh}

\newcommand{\M}{\mbox{\rm Mod-$R$}}

\newcommand{\lmla}{\mbox{\rm $\Lambda$-mod}}
\newcommand{\LMla}{\mbox{\rm $\Lambda$-Mod}}
\newcommand{\Mla}{\mbox{\rm Mod-$\Lambda$}}
\newcommand{\mla}{\mbox{\rm mod-$\Lambda$}}

\newcommand{\Z}{\mathbb{Z}}
\newcommand{\Q}{\mathbb{Q}}
\newcommand{\N}{\mathbb{N}}
\newcommand{\R}{\mathbb{R}}
\newcommand{\XX}{\mathbb{X}}

\DeclareMathOperator{\Hom}{Hom}
\DeclareMathOperator{\End}{End}
\DeclareMathOperator{\Ext}{Ext}
\DeclareMathOperator{\Tor}{Tor}
\DeclareMathOperator{\Ker}{Ker}

\DeclareMathOperator{\Coker}{Coker}

\newcommand{\A}{\ensuremath{\mathcal{A}}}
\newcommand{\B}{\ensuremath{\mathcal{B}}}
\newcommand{\C}{\ensuremath{\mathcal{C}}}

\newcommand{\Acal}{\ensuremath{\mathcal{A}}}
\newcommand{\Pcal}{\ensuremath{\mathcal{P}}}
\newcommand{\Ical}{\ensuremath{\mathcal{I}}}
\newcommand{\Rcal}{\ensuremath{\mathcal{R}}}
\newcommand{\Fcal}{\ensuremath{\mathcal{F}}}

\newcommand{\Dcal}{\ensuremath{\mathcal{D}}}
\newcommand{\Scal}{\ensuremath{\mathcal{S}}}
\newcommand{\Ucal}{\ensuremath{\mathcal{U}}}
\newcommand{\Tcal}{\ensuremath{\mathcal{T}}}
\newcommand{\Mcal}{\ensuremath{\mathcal{M}}}

\newcommand{\Ecal}{\ensuremath{\mathcal{E}}}
\newcommand{\Ccal}{\ensuremath{\mathcal{C}}}
\newcommand{\Qcal}{\ensuremath{\mathcal{Q}}}
\newcommand{\Bcal}{\ensuremath{\mathcal{B}}}

\newcommand{\T}{\ensuremath{\mathcal{T}}}
\newcommand{\F}{\ensuremath{\mathcal{F}}}

\newcommand{\fpA}{\ensuremath{\mathrm{fp}\,\mathcal{A}}}

\newcommand{\p}{\ensuremath{\mathbf{p}}}

\newcommand{\q}{\ensuremath{\mathbf{q}}}
\newcommand{\tube}{\ensuremath{\mathbf{t}}}

\newcommand{\Serre}{\ensuremath{\mathfrak{S}}}

\newcommand{\ifa}{if and only if}

\theoremstyle{plain}
\newtheorem{thm}{Theorem}[section]
\newtheorem{prop}[thm]{Proposition}

\newtheorem{lemma}[thm]{Lemma}
\newtheorem{cor}[thm]{Corollary}

\theoremstyle{definition}

\theoremstyle{remark}
\newtheorem{rem}{\bf Remark}

\newcommand{\La}{\Lambda}

\newcommand{\Dercat}{\ensuremath{D}}

\newcommand{\dba}{\Dercat(\Acal)}
\newcommand{\Db}{\mbox{$\mathcal \Dercat(\Mla)$}}
\begin{document}

\title[Tilting and cotilting modules over concealed canonical
algebras]{Tilting and cotilting modules\\ over concealed canonical
  algebras} 

\author[L. Angeleri H\"ugel]{Lidia Angeleri H\"ugel}
\address{Universit\`a degli Studi di Verona\\
 Strada Le Grazie 15 - Ca' Vignal 2\\
 I - 37134 Verona\\
 Italy}
\email{lidia.angeleri@univr.it}
\author[D. Kussin]{Dirk Kussin}
\address{{Instytut Matematyki \\ Uniwersytet Szczeci\'{n}ski \\ 70451
  Szczecin \\ Poland}}
\email{dirk@math.uni-paderborn.de}


\begin{abstract}
  We study infinite dimensional tilting modules over a concealed
  canonical algebra of domestic or tubular type. In the domestic case,
  such tilting modules are constructed by using the technique of
  universal localization, and they can be interpreted in terms of
  Gabriel localizations of the corresponding category of
  quasi-coherent sheaves over a noncommutative curve of genus zero. In
  the tubular case, we have to distinguish between tilting modules of
  rational and irrational slope. For rational slope the situation is
  analogous to the domestic case. In contrast, for any irrational
  slope, there is just one tilting module of that slope up to
  equivalence.  We also provide a dual description of infinite
  dimensional cotilting modules and a classification result for the
  indecomposable pure-injective modules.
\end{abstract}

\maketitle

\section{Introduction}
Infinite dimensional modules over canonical algebras have been studied
by several authors, see e.g.~\cite{R,RR,H,HP}. Particular attention
has been devoted to the problem of classifying the pure-injective
modules over canonical algebras of tubular type \cite{Rac, H, HP}. In
this paper, we approach the problem from the viewpoint of tilting
theory. Using the methods developed in \cite{AS2} for tame hereditary
algebras, we study the infinite dimensional tilting and cotilting
modules over a concealed canonical algebra $\La$ of domestic or
tubular type. The knowledge of these modules allows to obtain
classification results for the indecomposable pure-injective
$\La$-modules. It further enables us to reinterpret the classification
from \cite{AK} of the large quasi-coherent tilting sheaves over a
noncommutative curve of genus zero in terms of modules over a derived
equivalent algebra.

\smallskip

As in the hereditary case, the technique of universal localization due
to Cohn and Schofield plays a fundamental role in the construction of
tilting modules.  Indeed, if $\tube=\bigcup_{x\in\mathbb X} \Ucal_x$
is a stable, sincere, separating tubular family yielding a trisection
$(\p,\tube,\q)$ of the finite dimensional indecomposable
$\La$-modules, then any universal localization $\La\to\La_\Ucal$ at a
set $\Ucal$ of modules in $\tube$ is injective and gives rise to a
tilting module $\La_\Ucal\oplus \La_\Ucal/\La$.

\smallskip

If $\La$ has only homogeneous tubes, we obtain all but one infinite
dimensional tilting $\La$-modules in this way (up to equivalence). The
missing one is called Lukas tilting module, and it generates the class
$\Bcal$ of all modules without non-zero maps to the modules in $\p$.
More generally, if $\La$ has domestic representation type, large
tilting modules can have a more complicated shape involving also
finite dimensional summands, but the infinite dimensional part can
still be described in terms of universal localizations $\La_\Ucal$ and
Lukas tilting modules over such $\La_\Ucal$. We thus obtain a
classification of the large tilting, and by duality, of the large
cotilting modules, which is completely analogous to the tame
hereditary case treated in \cite{AS2,BK}, see Theorem
\ref{classdomestic}. In Corollaries \ref{onlysur} and \ref{Kronecker}
we also see that large tilting modules correspond to Gabriel
localizations of the category of quasi-coherent sheaves $\Qcoh\XX$
over the noncommutative curve of genus zero $\XX$ corresponding to
$\La$.

\smallskip

Let us remark that in the domestic case all large tilting or cotilting
modules are located in the ``central part'' of $\Mla$, that is, given
the trisection $(\p,\tube,\q)$, they admit neither maps to $\p$ nor
maps from $\q$. In other words, they belong to the intersection
$\Mcal=\Bcal\cap\Ccal$ of the torsion class $\Bcal$ of the torsion
pair $(\Bcal,\Pcal)$ in $\Mla$ generated by $\p$ with the torsion-free
class $\Ccal$ of the torsion pair $(\Qcal, \Ccal)$ cogenerated by
$\q$.

\smallskip

In the tubular case, the AR-quiver consists of a preprojective
component $\p_0$, a preinjective component $\q_\infty$, two non-stable
tubular families $\tube_0$ and $\tube_\infty$, and a countable number
of stable, sincere, separating tubular families $\tube_w, w\in \Q^+,$
giving rise to trisections $(\p_w,\tube_w,\q_w)$ of the finite
dimensional indecomposable $\La$-modules.  So, one can construct a
class $\Mcal_w=\Bcal_w\cap\Ccal_w$ as above for every $w\in\Q^+$.  For
irrational $w\in\R^+$, one divides the finite dimensional
indecomposable $\La$-modules into the class $\p_w$ of all modules that
belong to $\p_0$ or to one of the tubular families $\tube_v$ with
$v<w$, and the class $\q_w$ given by the remaining modules. The
corresponding torsion class $\Bcal_w$ and torsion-free class $\Ccal_w$
have an intersection $\Mcal_w$ which will only contain infinite
dimensional modules.

\smallskip

Following \cite{RR}, we say that the modules in $\Mcal_w$ have slope
$w$. This yields a notion of slope which is a module-theoretic
counterpart of the notion of slope for sheaves over a noncommutative
curve of genus zero, and which can be defined also for infinite
dimensional modules. Indeed, it is shown in \cite{RR} that every
indecomposable module (of finite or infinite dimension) has a slope.

\smallskip

For rational $w$, the classification of large tilting and cotilting
modules of slope $w$ is completely analogous to the domestic case.  In
contrast, for irrational $w$, there are just one tilting module
$\mathbf L_w$ and one cotilting module $\mathbf W_w$ of that slope.
Moreover, the modules of slope $w$ are precisely the pure-epimorphic
images of direct sums of copies of $\mathbf L_w$. Dually, they can be
described as the pure submodules of products of copies of
$\mathbf W_w$.  In particular, every indecomposable pure-injective
module of slope $w$ belongs to $\Prod \mathbf W_w$. Combining this
with results from \cite{H}, we obtain a classification of the
indecomposable pure-injective $\La$-modules in Theorem~\ref{Ziegler}.

\smallskip

Furthermore, in Theorem~\ref{limits}, we show that every tilting
module which does not belong to the leftmost part of $\Mla$ determined
by $\p_0\cup\tube_0$, nor to the rightmost part determined by
$\tube_\infty\cup\q_\infty$, has a slope. This yields a classification
of the tilting and cotilting modules in the ``central part'' of
$\Mla$.

\smallskip

Finally, let us turn to the category $\Qcoh\XX$ of quasi-coherent
sheaves over a noncommutative curve of genus zero $\mathbb X$. A
classification of the large tilting sheaves, with a self-contained
proof inside the hereditary category $\Qcoh\XX$, is given in
\cite{AK}. Here we illustrate the interplay between tilting sheaves
and tilting modules over a derived equivalent concealed canonical
algebra. More precisely, since by \cite[Lem.~7.10]{AK} every tilting
sheaf $\widehat{T}$ in $\Qcoh\XX$ is generated by a suitable tilting
bundle $T_{cc}$ in $\coh\XX$, we can regard $\Qcoh\XX$ as the heart of
a certain t-structure in the derived category of the algebra
$\Lambda=\End T_{cc}$ and $\widehat{T}$ as a tilting module in the
``central part'' of $\Mla$. This allows to recover the classification
of large tilting sheaves from \cite{AK}, see Theorem
\ref{classofsheaves}.  In particular, this shows that also the tilting
sheaves of rational slope can be described in terms of universal
localizations and Lukas tilting modules.

\medskip

The paper is organized as follows. After some preliminaries in
Sections \ref{prel} and \ref{prel2}, we discuss the construction of
tilting modules via universal localization in Section \ref{Uloc}.
Gabriel localizations of $\Qcoh\XX$ are treated in Section
\ref{Gabriel loc}. Section \ref{class results} is devoted to the
classification results mentioned above.

\medskip

{\bf Acknowledgements:} The authors would like to thank Helmut Lenzing
and Mike Prest for valuable discussions.  This research started while
the second named author was visiting the University of Verona with a
research grant of the Department of Computer Science. The first named
author is partially supported by Fondazione Cariparo, Progetto di
Eccellenza ASATA.

\bigskip

\section{Preliminaries}\label{prel}

Throughout this section, let $\Bcal$ be a Grothendieck category.
Given a class of objects $\Mcal\subset\Bcal$, we denote
$$\Mcal^o=\{B\in\B\mid\Hom_\A(M,B)=0 \text{ for all } M\in\Mcal\}$$ 
$$\Mcal^\perp=\{B\in\B\mid\Ext^1_\A(M,B)=0 \text{ for all } M\in\Mcal\}$$  
 The classes ${}^o\Mcal$ and ${}^\perp\Mcal$ are defined dually. 
If $\Mcal$ consists of a single object $M$, we write $M^o$, $M^\perp$, etc.

\medskip

\subsection{\sc Torsion pairs.} \label{torsion pairs}
Recall that a pair of classes  $(\Tcal,\Fcal)$ in $\B$ is a
\emph{torsion pair} if $\Tcal={}^o\Fcal$,
and $\Fcal=\Tcal^o$.
Every class $\Mcal$ of objects in $\B$ \emph{generates} a torsion pair
$(\Tcal,\Fcal)$ by setting $\Fcal=\Mcal^o$ and
$\Tcal={}^o(\Mcal^o)$. Similarly, the torsion pair \emph{cogenerated
  by} $\Mcal$ is given by $\Tcal={}^o\Mcal$ and $\Fcal=({}^o\Mcal)^o$.
We say that a torsion pair $(\Tcal,\Fcal)$ is \emph{split} if every
short exact sequence $0\to T\to M\to F\to 0$ with $T\in\Tcal$ and
$F\in \Fcal$ splits.  Finally, a torsion pair $(\Tcal,\Fcal)$ in $\B$
is {\em faithful} (or cotilting in the terminology of \cite{HRS,KST})
if the torsionfree class $\mathcal F$ generates the category
$\B$. Similarly, $(\mathcal T,\mathcal F)$ is {\it cofaithful} (or
tilting in the terminology of \cite{HRS,KST}) if the torsion class
$\mathcal T$ cogenerates $\B$.
   
\medskip

\subsection{\sc Gabriel localization.}\label{Gabrieloc}
Torsion pairs play an important role in connection with the
localization theory developed by Gabriel, see \cite{Ga,GL,P}.  First
of all, recall that a subcategory $\mathfrak S$ of $\Bcal$ is a {\em
  Serre subcategory} if for every exact sequence
$0\to A\to B\to C\to 0$ in $\Bcal$ we have $B\in\Serre$ if and only if
$A,C\in\Serre$. We can then form the quotient category $\Bcal/\Serre$
with the canonical quotient functor $q:\Bcal\to\Bcal/\Serre$.  A
torsion pair $(\Tcal,\Fcal)$ in $\B$ is said to be \emph{hereditary}
if the torsion class $\Tcal$ is closed under subobjects, or in other
words, $\Tcal$ is a Serre subcategory of $\Bcal$. If $\Bcal$ has
enough injectives, then $\Tcal$ is even a \emph{localizing}
subcategory of $\Bcal$, i.e. the quotient functor
$q:\Bcal\to\Bcal/\Tcal$ has a right adjoint.  More details will be
given in Theorem \ref{Gabriel}.

\medskip

\subsection{\sc Tilting and cotilting objects.}\label{tilting objects}
An object $V$ of $\B$ is {\em tilting} if the category $\Gen V$ of
$V$-generated objects equals $V^\perp$.
 Then $$(\Gen V,V^o)$$ is a cofaithful torsion pair in $\B$, and $\Gen V$ is called a {\em tilting class}.

As shown in \cite[2.1 and 2.2]{CF}, 
the equality $\Gen V=V^\perp$ is equivalent to the three conditions
\begin{enumerate}
\item[(T1)] pdim$V\le 1$, that is, the functor $\Ext^2_\B(V,-)=0$,
\item[(T2)] $\Ext^1_\B(V,V^{(\alpha)})=0$ for all cardinals $\alpha$,
\item[(T3)] an object $B\in\B$ is zero whenever it satisfies $\Hom_\B(V,B)=\Ext^1_\B(V,B)=0$.
\end{enumerate}

When $\B=\Mla$ for some ring $\Lambda$, condition (T3) can be rephrased as follows:
\begin{enumerate}
\item[(T3')] There is a short exact sequence
  $0\to \Lambda\to T_0\to T_1\to 0$ where $T_0,T_1\in\Add T$.
\end{enumerate}

\smallskip

{\em Cotilting objects} and {\em cotilting classes} in $\B$ are
defined dually. In particular, a module $W$ is cotilting if the
category $\Cogen W$ of $W$-cogenerated objects equals ${}^\perp W
$, or equivalently,  $C$ has the dual properties (C1)-(C3). Of course, the torsion pair
$$({}^oW,\Cogen W)$$ is then a faithful torsion pair in $\Mla$.  

\smallskip 

We will discuss classification of tilting and cotilting modules up to
equivalence. Hereby, we say that two tilting modules $T,T'$ are {\em
  equivalent} if they induce the same tilting class $\Gen T=\Gen T'$,
or equivalently, if they have the same additive closure
$\Add T=\Add T'$. Similarly, two cotilting modules $W,W'$ are {\em
  equivalent} if they induce the same cotilting class
$\Cogen W=\Cogen W'$, or equivalently, $\Prod W=\Prod W'$.

\smallskip

When $\Lambda$ is a left noetherian ring with a fixed duality $D$ (for
example $D=\Hom_\Z(-,\Q/\Z)$, or $D=\Hom_K(-,K)$ in case $\Lambda$ is
a finite dimensional algebra over a field K), tilting and cotilting
modules are related by the following result.

\begin{thm}[\cite{BH,APST}]\label{res} 
If $\Lambda$ is a left noetherian ring, there is a bijection between
\begin{enumerate} 
\item  equivalence  classes of  tilting modules in $\Mla$,
\item  equivalence  classes of  cotilting modules in $\LMla$,
\item  resolving subcategories of $\mla$ consisting of modules of projective dimension $\le 1$.
\end{enumerate}
The bijection above assigns to a tilting module $T$ the cotilting
module $D(T)$, and the resolving subcategory
$\Scal={}^\perp(\Gen T)\cap\mla$.  Conversely, a resolving subcategory
$\Scal$ is mapped to the tilting class $\Scal^\perp$ and the cotilting
class
$\Scal^\intercal= \{B\in\LMla\mid\Tor_1^\Lambda (S,B)=0 \text{ for all
} S\in\Scal\}$.\end{thm}

Hereby, a subcategory $\Scal\subset\mla$ (or of $\Mla$) is said to be
{\em resolving} if it is closed under direct summands, extensions, and
kernels of epimorphisms, and it contains $\Lambda$ (or all projective
modules, respectively).

\smallskip

(Co)tilting modules that are not equivalent to a finitely generated
(co)tilting module will be called {\it large}.

\medskip

\subsection{\sc The heart.}\label{heart}

Let $(\Qcal, \Ccal)$ be a torsion pair in $\Bcal$.  According to
\cite{HRS, KST}, the classes
$$\mathcal D^{\le 0}=\{X^\cdot\in\mathcal \Dercat(\Bcal)\mid
H^0(X^\cdot)\in\mathcal Q, H^i(X^\cdot)=0 \text{ for } i>0\},$$
$$\mathcal D^{\ge 0}=\{X^\cdot\in\mathcal \Dercat(\Bcal)\mid
H^{-1}(X^\cdot)\in\mathcal C, H^i(X^\cdot)=0 \text{ for } i<-1\}$$ 
form a t-structure $(\mathcal D^{\le 0},\mathcal D^{\ge 0})$ in the
derived category $\mathcal \Dercat(\Bcal)$, called the {\it
  t-structure induced by $(\mathcal Q,\mathcal C)$}. Its \emph{heart}
$$\mathcal A=\mathcal D^{\le 0}\cap\mathcal D^{\ge 0}$$   is always an
abelian category \cite{BBD}  whose exact structure is given by the
triangles of $\mathcal \Dercat(\Bcal)$. For any two objects $X,Z\in\A$
there are functorial isomorphisms $$\Ext^i_\A(X,Z)\cong\Hom_{\mathcal
  \Dercat(\Bcal)}(X,Z[i])\text{ for } i=0,1.$$  

Moreover, $(\mathcal C[1],\mathcal Q)$ is a torsion pair in $\A$ by
\cite[I.2.2]{HRS}.

\smallskip {\it From now on}, we assume that $(\mathcal Q,\mathcal C)$
is a {faithful torsion pair} and $\A$ is the heart of the
corresponding t-structure in $\mathcal \Dercat(\Bcal)$. Then there is
a triangle equivalence between $\mathcal \Dercat(\Bcal)$ and $\dba$,
see e.g.~\cite[3.12]{KST}. We record the following facts for later
reference.

\begin{lemma}\label{ExtHom}
  \textnormal{(1)} For every $X\in\A$ there are objects $Y\in\Ccal$
  and $Q\in\Qcal$ with a canonical sequence
  $$0\to Y[1]\to X\to Q\to 0.$$ \textnormal{(2)} The following
  statements hold true for $C,Y\in\Ccal$ and $Q\in\Qcal$.
\begin{enumerate}
\item[(a)] $\Hom_\A(Q,C[1])\cong\Ext_\B^1(Q,C)$,
\item[(b)] $\Ext^1_\A(Q,C[1])\cong\Ext_\B^2(Q,C)$,
\item[(c)] $\Hom_\A(Y[1],C[1])\cong\Hom_\B(Y,C)$,
\item[(d)] $\Ext^1_\A(Y[1],C[1])\cong\Ext^1_\B(Y,C)$,
\item[(e)] $\Ext^1_\A(Y[1],Q)\cong\Hom_\B(Y,Q)$.
\end{enumerate}
 \end{lemma}
\begin{proof} is left to the reader.\end{proof}

\smallskip

As shown in \cite[5.2]{KST}, the heart $\A$ is a \emph{hereditary}
abelian category, that is, $\Ext^2_\A(-,-)=0$, if and only if the
torsion pair $(\Qcal, \Ccal)$ is split and all objects in $\Ccal$ have
projective dimension at most one.  In this case, the (co)tilting
objects in $\A$ and $\B$ are closely related.
  
 \begin{lemma}\label{tilting in hearts}
Assume that $\A$ is hereditary.
\begin{enumerate}
\item[(1)] An object $C\in\Bcal$ of injective dimension at most one
  that belongs to $\Ccal$ satisfies conditions (C2) and (C3) in $\B$
  if and only if so does $C[1]$ in $\A$.
  
\item[(2)]  An object $T\in\Bcal$ that belongs to $\Ccal$ 
 satisfies conditions (T2) and (T3) in $\B$ if and only if so does $T[1]$  in $\A$. 
\end{enumerate}
  \end{lemma}
\begin{proof}
  It follows from Lemma~\ref{ExtHom}(d) that $C$ satisfies condition
  (C2) in $\Bcal$ if and only if so does $C[1]$ in $\A$, and similarly
  $T$ satisfies condition (T2) in $\Bcal$ if and only if so does
  $T[1]$ in $\A$.

  Assume now that $C$ satisfies (C3), and let $X\in\A$ be an object
  satisfying $\Hom_\A(X,C[1])=\Ext^1_\A(X,C[1])=0$. From the canonical
  sequence $$0\to Y[1]\to X\to Q\to 0$$ with $Y\in\Ccal$ and
  $Q\in\Qcal$ we obtain a long exact sequence
  $$0\to \Hom_\A(Q,C[1])\to 0\to \Hom_\A(Y[1],C[1])\to
  \Ext^1_\A(Q,C[1])\to 0\to \Ext^1_\A(Y[1],C[1])\to 0.$$
  We infer from Lemma~\ref{ExtHom}(a) that $\Ext^1_\B(Q,C)=0$, and as
  $\Hom_\B(Q,C)=0$ by assumption, we conclude $Q=0$. Moreover, it
  follows from Lemma~\ref{ExtHom}(b) and (c) that
  $\Hom_\B(Y,C)\cong \Ext^1_\A(Q,C[1])\cong\Ext^2_\B(Q,C)=0$ as $C$
  has injective dimension at most one. Since we also have
  $\Ext^1_\B(Y,C)=0$ by Lemma~\ref{ExtHom}(d), we conclude $Y=0$, and
  so $X=0$. This shows that $C[1]$ satisfies condition (C3).

  Conversely, assume that $C[1]$ satisfies condition (C3), and let
  $B\in\mathcal B$ be an object satisfying
  $\Hom_{\mathcal B}(B,C)=\Ext^1_{\mathcal B}(B,C)=0$. Then
  $B=Y\oplus Q$ with $Y\in\Ccal$ and $Q\in\mathcal Q$. By
  Lemma~\ref{ExtHom}(c) and (d) it follows that
  $\Hom_\A(Y[1],C[1])=\Ext^1_\A(Y[1],C[1])=0$, hence $Y=0$. Further,
  again from Lemma~\ref{ExtHom}(b) we obtain
  $\Ext^1_\A(Q,C[1])=0$. Since $\Hom_\A(Q,C[1])=0$ by
  Lemma~\ref{ExtHom}(a), we conclude $Q=0$ and $B=0$. So $C$ satisfies
  (C3) as well.

The proof of statement (2) is dual. 
\end{proof}

\medskip

\subsection{\sc Hearts induced by  cotilting modules.}\label{hearts of cotilting}
Assume now $\mathcal B=\Mla$ for some ring $\Lambda$.
Then, 
as shown in \cite[Sections 3 and 4]{CF}, the object
$V=\Lambda[1]\in\Acal$ is a tilting object with
$\End_\A V\cong\Lambda$,
defining crosswise equivalences
$$H_V=\Hom_\Acal(V,-):\Ccal[1]=\Gen V\to \Ccal,\quad
H_V'=\Ext^1_\Acal(V,-):\mathcal Q=V^o\to\mathcal Q$$
between the torsion class in $\Acal$ and the torsionfree class in
$\Mla$, and between the torsion-free class in $\A$ and the torsion
class in $\Mla$, respectively. In other words,
${\mathbb R} \Hom_\A(V,-)$ yields an equivalence $\dba\to\Db$.

\begin{prop}\label{shape}
Let $(\Qcal, \Ccal)$ be a faithful 
torsion pair in $\Mla$ with hereditary heart $\A$.
\begin{enumerate}
\item[(1)] If $C\in\Ccal$ is a cotilting $\Lambda$-module, then
  $\Cogen C[1]={}^\perp C[1]$ is a torsionfree class in $\A$
  consisting of the objects in $X\in\A$ for which $H_V(X)\in\Cogen
  C$.
  In particular, $\mathcal Q$ is always contained in $\Cogen C[1]$.

\item[(2)] If $T\in\Ccal$ is a tilting $\Lambda$-module, then
  $\Gen T[1]=T[1]^\perp$ is a torsion class in $\A$ consisting of the
  objects in $X\in\A$ for which $H_V(X)\in\Gen T$ and
  $H_V'(X)\in T^o$. \end{enumerate}
\end{prop}
\begin{proof}
Since $\A$ is hereditary,  conditions (C1) and (T1) are always satisfied in $\A$. 

(1) We know from Lemma \ref{tilting in hearts} that $C[1]$ satisfies
(C1)-(C3). Dualizing the proof of \cite[2.1]{CF}, we infer that
$\Cogen\, C[1]={}^\perp C[1]$ is a torsionfree class in $\A$. Let
$X\in\A$. Taking again the canonical sequence
$0\to Y[1]\to X\to Q\to 0$ with $Y\in\Ccal$ and $Q\in\Qcal$ and
recalling that $\Ext^1_\A(Q,C[1])\cong \Ext^2_{\Lambda}(Q,C)=0$ and
$\Ext^1_\A(Y[1],C[1])\cong\Ext^1_{\Lambda}(Y,C)$ by
Lemma~\ref{ExtHom}(b) and (d), we see that $X\in\Cogen\,C[1]$ if and
only if $Y\in\Cogen C$. The claim now follows from the fact that
$Y=H_V(X)$ and $H_V(Q)=0$ for all $Q\in\mathcal Q$.

(2) Again we infer from Lemma \ref{tilting in hearts} and the proof of
\cite[2.1]{CF} that $\Gen\, T[1]=T[1]^\perp$ is a torsion class in
$\A$. Applying $\Hom_\A(T[1],-)$ to the canonical sequence
$0\to Y[1]\to X\to Q\to 0$ with $Y\in\Ccal$ and $Q\in\Qcal$, we obtain
a long exact sequence
$0= \Hom_\A(T[1],Q)\to \Ext^1_\A(T[1],Y[1])\to \Ext^1_\A(T[1],X)\to
\Ext^1_\A(T[1],Q)\to 0$.
By Lemma~\ref{ExtHom}(d) and (e), we see that $\Ext^1_\A(T[1],X)=0$ if
and only if $Y\in\Gen T$ and $\Hom_{\Lambda}(T,Q)=0$. The claim now
follows from the fact that $Y=H_V(X)$ and $Q=H_V'(X)$.
\end{proof}

\smallskip
 
 We will be particularly interested in the case when 
 $\A$ is a \emph{Grothendieck category}. It was shown in \cite{CGM}
 that this happens if and only if there is a cotilting module $W$ such
 that $\mathcal C=\Cogen W$ (and $\mathcal Q={}^o W$). Then $W[1]$ is
 an injective cogenerator of $\A$.  In some cases, $\A$ has also the
 following geometric interpretation.

\begin{prop}\label{sheavesashearts}
  Let $\Lambda$ be a connected artin algebra, and let
  $(\mathcal Q,\mathcal C)$ be a torsion pair in $\Mla$. Suppose that
\begin{enumerate}
\item[(i)] there is a $\Sigma$-pure-injective cotilting
  $\Lambda$-module $W$ such that $\mathcal C=\Cogen W$,
\item[(ii)] the torsion pair $(\mathcal Q,\mathcal C)$ splits,
\item[(iii)] $\Lambda\in\mathcal C$ and $D(_\Lambda\Lambda)\in\mathcal Q$, 
\item[(iv)] $\mathcal Q\cap{}^\perp\mathcal Q=0$, 
\item[(v)]  all modules in $\Ccal$ have projective dimension at most one.
\end{enumerate}
Then the heart $\mathcal A$ of the corresponding t-structure in $\Db$
is equivalent to the category $\Qcoh\XX$ of quasi-coherent sheaves
over a noncommutative curve of genus zero $\mathbb X$.  Hereby the
category fp$\Acal$ of finitely presented objects corresponds to the
category $\coh\XX$ of coherent sheaves.
\end{prop}
\begin{proof}
  First of all, the $\Sigma$-pure-injectivity of $W$ yields by
  \cite{CMT} that $\mathcal A$ is a locally noetherian Grothendieck
  category.  Moreover, since the torsion pair
  $(\mathcal Q,\mathcal C)$ splits and all modules in $\Ccal$ have
  projective dimension at most one, it follows from \cite[5.2]{KST}
  that $\mathcal A$ is a hereditary category, that is,
  $\Ext^2_\A(-,-)=0$.

  Set $\mathcal H={\rm fp}\A$. Then $\mathcal H$ is a noetherian
  hereditary category with tilting object $V=\Lambda[1]$. The objects
  of $\mathcal H$ are extensions of objects of the form $Y[1]$ with
  $Y\in\mathcal C\cap\mla$ by objects $Q\in\mathcal Q\cap\mla$. It is
  then clear that $\mathcal H$ is a connected, skeletally small
  abelian $k$-category for which all morphism and extension spaces are
  finite dimensional $k$-vector spaces.
 
  We claim that $\mathcal H$ has no non-zero projective
  objects. Assume that $A\in\mathcal H$ is a non-zero projective
  object with canonical exact sequence $0\to Y[1]\to A\to Q\to
  0$.
  Then also $Y[1]$ is projective because $\mathcal H$ is
  hereditary. Since $\mathcal Q$ contains an injective cogenerator of
  $\Mla$, the condition
  $\Ext^1_{\mathcal H}(Y[1],Q')\cong\Hom_\Lambda(Y,Q')=0$ for all
  $Q'\in\mathcal Q$ implies that $Y[1]=0$ and $A\cong Q\in\mathcal Q$.
  By condition (iv) there is a non-split short exact sequence
  $0\to Q'\to B\stackrel{g}{\to} Q\to 0$ in $\Mla$ with all terms in
  $\mathcal Q$.  We show that the sequence is also exact in the heart
  $\A$.  To this end, we apply \cite[pp.281]{GM} to compute the kernel
  and cokernel of $g:B\to Q$ viewed as a morphism in $\A$: if $Z$ is
  the cone of $g$ in $\Db$ and
  $K=\tau_{\le-1}Z\to Z\to \tau_{\ge 0}Z\to K[1]$ is the canonical
  triangle where $\tau_{\le-1}Z\in\mathcal D^{\le -1}$ and
  $\tau_{\ge 0}Z\in\mathcal D^{\ge 0},$ then $\Ker_\A (g)=K[-1],$ and
  $\Coker_\A (g)=\tau_{\ge 0}Z$.  But $Z$ has homologies
  $Q'\in\mathcal Q$ in degree -1 and 0 elsewhere, thus $Z\cong K$, and
  $g$ is an epimorphism in $\A$ with kernel $Q'$.  So we have a
  non-split exact sequence in $\A$ ending at the projective object
  $Q$, a contradiction.

  By the axiomatic description of $\coh\XX$ given in \cite[2.5]{LMik},
  we now conclude that $\mathcal H=\coh\XX$ for a noncommutative curve
  of genus zero $\mathbb X$. Hereby $\mathbb X$ is obtained as an
  index set when decomposing the category $\mathcal H_0$ of
  indecomposable finite length objects of $\mathcal H$ into a family
  of connected uniserial length categories
  $\mathcal H_0=\bigcup_{x\in\mathbb X} \mathcal U_x$. Finally
  $\mathcal A=\Qcoh\XX$, cf. \cite[Ch. VI]{Ga} or \cite[5.4]{CB}.
 \end{proof}
 
 We  can now improve Proposition \ref{shape} as follows.
 
 \begin{cor}\label{shape2}
   Under the assumptions of Proposition \ref{sheavesashearts} above, a
   $\Lambda$-module $T\in\Ccal$ is a tilting module if and only if
   $T[1]$ is a tilting object in $\A$. In this case,
   $\Gen T[1]=\{X\in\A\,\mid\, H_V(X)\in\Gen T \text{ and }
   H_V'(X)=0\}$.
   In particular, $\Gen T[1]$ is contained in $\mathcal C[1]$.
\end{cor}
\begin{proof}
  Observe first that $T\in\C$ and $T[1]\in\A$ have projective
  dimension at most one. Moreover, both $\Mla$ and $\A$ have enough
  injectives. The first statement then follows immediately from
  Proposition \ref{tilting in hearts}.

  By Proposition \ref{shape}, $\Gen T[1]$ consists of the $X\in\A$ for
  which $H_V(X)\in\Gen T$ and $H_V'(X)\in T^o$, that is,
  $\Hom_\Lambda(T,H_V'(X))=0$.  But $H_V'(X)$ is a module from
  $\mathcal Q$, and since $T\in\mathcal C\subset {}^\perp\mathcal Q$
  by assumption (ii), we always have
  $\Ext^1_{\Lambda}(T,H_V'(X))=0$. By condition (T3) for the tilting
  module $T$ we get that $H_V'(X)=0$ whenever $X\in\Gen T[1]$. So
  $\Gen T[1]\subset\mathcal C[1]$.
\end{proof}

\bigskip

\section{Concealed canonical algebras}\label{prel2}

\subsection{\sc The setup.}\label{setup}
{\it From now on} $\Lambda$ denotes a finite dimensional, connected,
concealed canonical algebra over a field $k$, for example a tame
hereditary algebra, or a canonical algebra. By \cite{LP} concealed
canonical algebras are precisely the finite dimensional algebras with
a sincere stable separating tubular family
$\tube=\bigcup_{x\in\mathbb X} \Ucal_x$ yielding a canonical
trisection $$(\p,\tube,\q)$$ of the category $\mla$.

More precisely, $\tube$ is a family of standard tubes $\Ucal_x$ in the
Auslander-Reiten quiver of $\La$ which is

- {\it sincere}: every simple module occurs as the composition factor
of at least one module from $\tube$;

- {\it stable}: it does not contain indecomposable projective or injective modules;

- \emph{separating}: the indecomposable modules in $\mla$ that do not
belong to $\tube$ fall into two classes $\p$ and $\q$ such that
$\Hom(\q,\p)=\Hom(\q,\tube)=\Hom(\tube,\p)=0$, and any homomorphism
from a module in $\p$ to a module in $\q$ factors through any
$\Ucal_x$.

\smallskip
The modules in $\add\tube$ form an exact abelian subcategory of $\mla$
in which all objects have finite length. The simple objects and the
composition factors in this category will be called {\it simple
  regular} modules and {\it regular composition factors}. The set of
all simple regular modules in a tube $\Ucal_x$ is called the {\it
  clique} of $\Ucal_x$. The order of the clique is the {\it rank} of
$\Ucal_x$. Notice that almost all tubes are {\it homogeneous}, i.e.~of
rank one.

Every simple regular module $S=S_1\in\Ucal_x$ determines a {\it ray}
$\{S_n\mid n\in\N\}$ of $\Ucal_x$, where $S_n$ denotes the
indecomposable object of regular length $n$ with regular socle
$S$. The direct limit of the modules on a ray
$S_\infty=\varinjlim S_n$ is called {\it Pr\"ufer module}, the {\it
  adic module} $S_{-\infty}$ is defined dually. Both are
indecomposable, infinite dimensional, pure-injective modules.  By
abuse of terminology, we say that $S_\infty$, or $S_{-\infty}$, is a
{Pr\"ufer module}, respectively an adic module, {\it from the tube}
$\Ucal_x$.

Given a tube $\Ucal_x$ of rank $r>1$ and a module $S_m\in\Ucal_x$ of
regular length $m<r$, we consider the full subquiver
$\mathcal W_{S_m}$ of $\Ucal_x$ which is isomorphic to the
Auslander-Reiten-quiver $\Theta(m)$ of the linearly oriented quiver of
type $\mathbb A_m$ with $S_m$ corresponding to the
projective-injective vertex of $\Theta(m)$. The set $\mathcal W_{S_m}$
is called a \emph{wing} of $\Ucal_x$ of size $m$ with {\it vertex}
$S_m$.

\smallskip

It is shown in \cite[3.1 and \S 10]{RR} that the class $\q$ generates
a split torsion pair $$(\Gen\q,\mathcal C)$$ in $\Mla$ and that
$\mathcal C=\Cogen \mathbf{W}$ for a cotilting module $\mathbf{W}$
which is the direct sum of of all Pr\"ufer modules $S_\infty$, where
$S$ runs through the isoclasses of all simple regular
$\Lambda$-modules in $\tube$, and an indecomposable infinite
dimensional module $G$ which has finite length over its endomorphism
ring and is called the {\it generic} module. Note that in the tame
hereditary case $\Gen\q=\Add\q$ and $\mathcal C$ is the largest
cotilting class in \Mla\ which is induced by a large cotilting module
(cf.\cite[\S 2]{AS2}).

\medskip

We consider the t-structure induced by the torsion pair
$(\Gen\q,\mathcal C)$ in \Db, and denote its heart by $\mathcal A$. We
claim that $\A$ is equivalent to the category $\Qcoh\XX$ of
quasi-coherent sheaves over $\mathbb X$.

\medskip

Indeed, $\mathbf W$ is a $\Sigma$-pure-injective cotilting module, and
we infer as above that $\mathcal A$ is a hereditary locally noetherian
Grothendieck category with injective cogenerator $\mathbf{W}[1]$, see
also \cite[11.1]{RR}.  The indecomposable injective objects in
$\mathcal A$ are $G[1]$ and the objects $S_\infty[1]$ where $S$ runs
through the isoclasses of all simple regular $\Lambda$-modules. Notice
that $S_\infty[1]$ is a uniserial object with socle $S[1]$, and
$\mathcal H_0=\tube[1]$ is the category of indecomposable finite
length objects in $\mathcal A$.

\medskip
Of course $\Lambda\in\mathcal C={}^\perp \mathbf W$, and
$D(\Lambda_\Lambda)\in\mathcal \Gen\q$ since $\tube$ is
stable. Further, all modules in $\Ccal$ have projective dimension at
most one by \cite[5.4]{RR}.  Finally,
${}^\perp(\Gen\q)\subset{}^\perp\q=\Ccal$, hence
${}^\perp(\Gen\q)\cap\Gen \q=0$.  So conditions (i) - (v) in
Proposition \ref{sheavesashearts} are satisfied, and we deduce that
$\mathcal A$ is equivalent to the category of quasi-coherent sheaves
over a noncommutative curve of genus zero, which coincides with
$\mathbb X$ because
$\mathcal H_0=\tube[1]=\bigcup_{x\in\mathbb X} \Ucal_x[1]$.

\medskip

The indecomposable finitely presented objects of infinite length in
$\A$ form the class $\vect\mathbb X=\q\cup\p[1]$ of indecomposable
{\it vector bundles}. Notice that $\tube$ generates the torsion pair
$(\Gen\tube, \varinjlim\p)$ in $\Mla$ with torsion-free class
$\varinjlim\p=\Cogen G$ by \cite[3.5 and 6.6]{RR}, and $\mathcal H_0$
generates a torsion pair
$(\varinjlim\tube[1],\varinjlim\vect\mathbb X)$ in $\A$. We call a
module or a sheaf {\it torsion}, respectively {\it torsion-free}, if
it is torsion, respectively torsion-free, with respect to these
torsion pairs.  Finally, $S_\infty[1]$, $G[1]$, $S_{-\infty}[1]$ are
called {\it Pr\"ufer, generic, adic sheaves}, and {\it wings} in the
tubular family $\mathcal H_0$ are defined in analogous way as above.

\medskip

\subsection{\sc Representation type.} According to \cite[Theorem
7.1]{LP}, a numerical invariant called genus determines the
representation type of the algebra $\La$, which can be {\it domestic,
  tubular} or {\it wild}. In the domestic case, $\Lambda$ is {\it tame
  concealed}, i.~e.~it can be realized as endomorphism ring of a
preprojective or preinjective tilting module over a finite dimensional
tame hereditary algebra $\Lambda'$.  The tubular case will be
discussed in more detail in Section \ref{class results}.

\subsection{\sc The Auslander-Reiten formula}\label{ARF}
Denote by $(\Dcal,\Rcal)$ the torsion pair cogenerated by $\tube$. As
shown in \cite{RR}, it is a split torsion pair, and the module
$\mathbf W$ considered in \ref{setup} is a tilting module whose
tilting class is the class of {\it divisible modules}
$\Dcal=\Gen\mathbf W$.  By \cite[10.1]{RR}
$$\mathcal C\cap\mathcal D=\Add \mathbf W=\Prod\mathbf W.$$

\begin{lemma}\cite[5.4]{RR}\label{5.4} The modules in $\C$ have
  projective dimension at most one, the modules in $\Dcal$ have
  injective dimension at most one.
\end{lemma}

In particular, the modules in $\p$ have projective dimension at most
one, while those in $\q$ have injective dimension at most one. We will
frequently use the following version of the Auslander-Reiten formula
without further reference.

\begin{lemma}\cite{Sto}\label{Sto}
  Let $A,C$ be $\La$-modules, and assume that $A$ is finitely
  generated without non-zero projective summands.
\begin{enumerate}
\item If ${\rm pdim}A\le 1$,  then ${\Hom}_\Lambda\,(C,\tau\,A) \cong D\Ext^1_\Lambda\,(A,C)$.
 \item If ${\rm idim}\tau A\le 1$, then $D\,{\Hom}_\Lambda\,(A,C)\cong \Ext^1_\Lambda\,(C,\tau\,A)$.
\end{enumerate}
\end{lemma}

As a first application, we see that $\Dcal={}^o\tube=\tube^\perp$ and
$\Ccal=\q{}^o={}^\perp\q$. Further, we consider the torsion pair
$(\Bcal, \Pcal)$ in $\Mla$ cogenerated by the class $\p$. Then
$\Bcal={}^o\p=\p^\perp$, and by Theorem \ref{res} there is a tilting
module $\mathbf L$ with tilting class $\Gen \mathbf L=\Bcal$. By
\cite[2.2]{KT}, $\mathbf L$ has an infinite filtration by modules in
$\p$, so in particular it is torsion-free and belongs to $\Ccal$. We
call it the {\it Lukas tilting module} as its construction in the
hereditary case goes back to \cite{Lu}, cf.~\cite[3.3]{KT}.
 
\begin{lemma}\label{0}
  Let $X$ be a Pr\"ufer module, or an adic module, or the generic
  module. Further, let $P\in\mathcal P$ and $Q\in\mathcal Q$ be
  non-zero modules. Then $\Ext^1_\Lambda\,(Q,X)\not=0$ and
  $\Ext^1_\Lambda\,(X,P)\not=0$.
\end{lemma}
\begin{proof}
  First of all, one shows as in \cite[2.5]{BK} that
  $\Ext^1_\Lambda\,(Q,X)\not=0$ when $Q\in\q$, and
  $\Ext^1_\Lambda(X,P)\not=0$ when $P\in\p$.

  Assume now $\Ext^1_\Lambda\,(Q,X)=0$. Notice that $X$ has injective
  and projective dimension at most one. When $X$ is a Pr\"ufer module
  or the generic module, this follows from Lemma \ref{5.4}, and for
  adic modules it follows by duality.  Hence
  $\Ext^1_\Lambda\,(Q',X)=0$ for all submodules $Q'$ of $Q$. So $Q$
  cannot have submodules in $\q$ and therefore all its finitely
  generated submodules lie in $\Ccal$. But then
  $Q\in\varinjlim\Ccal=\Ccal$, a contradiction.

  For the second statement we proceed dually. If
  $\Ext^1_\Lambda\,(X,P)=0$, it follows that
  $\Ext^1_\Lambda\,(X,P')=0$ for all quotients $P'$ of $P$. So all
  finitely generated factor modules of $P$ lie in $\Bcal$. Since every
  module can be purely embedded in the direct product of all its
  finitely generated factor modules, see \cite[2.2. Ex 3]{CB2}, and
  $\Bcal$ is closed under direct products and pure submodules, we
  infer $P\in\Bcal$, a contradiction.
\end{proof}

\medskip

\subsection{\sc Purity.}\label{purity}
The objects of fp$\A=\coh\XX$ are pure-injective. Indeed, if
$\varepsilon: 0\to A\to B\to C\to 0$ is a pure exact sequence in $\A$
and $X\in\A$ is finitely presented, then
$\Hom_\A(\tau^-X,\varepsilon)$ is exact. Since $\Ext^2_\A(-,-)$
vanishes, this amounts to exactness of
$\Ext^1_\A(\tau^-X,\varepsilon)$, which in turn is equivalent to
exactness of $D\Hom_\A(\varepsilon, X)$ by Serre duality. But this
means that $\Hom_\A(\varepsilon, X)$ is exact, which gives the claim.

\smallskip

Furthermore, a module $C\in\Ccal$ is pure-injective if and only if
$C[1]$ is a pure-injective object in $\A$. This follows from the
following criterion by Jensen and Lenzing.

\begin{lemma}\cite[Theorem 5.4]{P} \label{JL}
  An object $A$ in a locally noetherian category is pure-injective if
  and only if the summation map $A^{(I)}\to A$ factors through the
  canonical embedding $A^{(I)}\to A^I$ for every set $I$.
\end{lemma}

\begin{prop} \label{pi} 
Assume $\Lambda$ has domestic representation type.
\begin{enumerate}
\item The indecomposable pure-injective objects in $\A=\Qcoh\XX$ are
  precisely the indecomposable coherent objects, the Pr\"ufer sheaves,
  the adic sheaves, and the generic sheaf.
\item The indecomposable pure-injective $\Lambda$-modules are
  precisely the finite dimensional indecomposable modules, the
  Pr\"ufer modules, the adic modules and the generic module.
\end{enumerate}
\end{prop}
\begin{proof}
  (1) By assumption on $\Lambda$, there is a finitely presented
  tilting object $V\in\A=\Qcoh\XX$ inducing a derived equivalence
  $\mathcal \Dercat({\A})\to \mathcal \Dercat({\Mla'})$ for a tame
  hereditary algebra $\Lambda'$. Denote by $\q'$ the preinjective
  component of $\Lambda'$, and by $(\Add\q',\Ccal')$ the corresponding
  split torsion pair in $\Mla'$. Then $(\Ccal'[1],\Add\q')$ is a split
  torsion pair in $\A$, because the corresponding heart $\Mla'$ is a
  hereditary category. So, every indecomposable non-coherent
  pure-injective object $A\in\A$ belongs to $\Ccal'[1]$, and the claim
  follows from the discussion above and the well-known classification
  of pure-injective modules over tame hereditary algebras (see
  e.g.~\cite{CB2}).

  (2) We first show that $\Gen\q=\Add\q$, as in the hereditary
  case. Recall that $\Gen\q$ is the direct limit closure of $\q$ by
  \cite{CB} or \cite[4.5.2]{GT}. So every $Q\in\Gen\q$ has the form
  $Q=\varinjlim Q_i$ for a suitable system $(Q_i)_{i\in I}$ from $\q$,
  and there is a pure-exact sequence
  $\varepsilon:0\to K \to \bigoplus_{i\in I}Q_i\to Q\to 0$. Let
  $\overline{K}\in \Ccal$ be the torsion-free part of $K$, which is a
  direct summand as the torsion pair $(\Gen\q,\Ccal)$ splits. Then,
  since $\Ccal$ is definable, the pure-injective envelope of
  $\overline{K}$ lies in $\Ccal$ and factors through the
  pure-monomorphism $K\to \bigoplus_{i\in I}Q_i$. This shows that
  $\overline{K}=0$ and $\varepsilon$ has all terms in $\Gen\q$. As in
  the proof of Proposition \ref{sheavesashearts}, we infer that
  $\varepsilon$ is exact in $\A$, and one easily checks that it is
  even pure-exact.  Recall that the tilting object $V$ induces a
  functor $H_V=\Hom_\A(V,-)$ with kernel $\Add\q'$. Now
  $H_V(\varepsilon)$ is exact, thus $H_V(Q)=0$, showing that
  $Q\in\Add\q'$ is a direct sum of coherent objects in $\A$. Viewed as
  a $\Lambda$-module, $Q$ is then a direct sum of finitely presented
  modules in $\Gen\q$, thus $Q\in\Add\q$.

  Now we infer that every indecomposable infinite dimensional
  pure-injective module $X$ must belong to $\Ccal$, hence $X[1]$ is an
  indecomposable pure-injective object in $\A$, and the claim follows
  from (1).
\end{proof}

\bigskip

\section{Universal localization.}\label{Uloc}
In this section, we review the technique of universal localization
developed by Cohn and Schofield, which is needed for the construction
of tilting modules.

\begin{thm} [\cite{Scho}]\label{Scho}
Let $R$ be a ring. For any    set of morphisms $\Sigma$ between
finitely generated projective right $R$-modules
there is a ring homomorphism 
$\lambda\colon R\rightarrow R_\Sigma$ such that
\begin{enumerate}
\item $\lambda$ is \emph{$\Sigma$-inverting:}  if
$\alpha\colon P\rightarrow Q$ belongs to  $\Sigma$, then the $R_\Sigma$-homomorphism
$\alpha\otimes_R 1_{R_\Sigma}\colon P\otimes_R R_\Sigma\rightarrow
Q\otimes_R R_\Sigma$ is an isomorphism. 
\item $\lambda$ is \emph{universal} with respect to (1): any further
  $\Sigma$-inverting ring homomorphism $\lambda':R\to R'$ factors
  uniquely through $\lambda$.
\end{enumerate}
The homomorphism
$\lambda\colon R\rightarrow R_\Sigma$ is a ring epimorphism
with  $\Tor_1^{R}({R_\Sigma},{R_\Sigma})=0$, called
the \emph{universal localization} of $R$ at
$\Sigma$.
\end{thm}
Let now ${\mathcal{E}}\subset\mla$ be a set of modules of projective
dimension one.  For each $E\in\mathcal{E},$ we fix a projective
resolution $0\to P\stackrel{{\alpha_E}}{\to} Q\to E\to 0$ in $\mla$,
and we set $\Sigma=\{\alpha_E\mid E\in\mathcal{E}\}.$ We denote by
{$\Lambda_{\mathcal{E}}$} the universal localization of $\Lambda$ at
$\Sigma$, which does not depend on the chosen class $\Sigma$ by
\cite[Theorem~0.6.2]{Cohn}.

\begin{thm}\label{localizingwide}  Let $\mathcal U $ be a set of simple regular modules. 
  Then there is a short exact sequence
  $$0\to \Lambda\stackrel{\lambda}{\to} \Lambda_\Ucal\to
  \Lambda_\Ucal/\Lambda\to 0$$ where
  \begin{enumerate}
  \item $\lambda$ is a homological ring epimorphism,
    i.e.~$\Tor_i^\Lambda(\Lambda_\Ucal,\Lambda_\Ucal)=0$ for all
    $i>0$,
  \item $\Ucal^\wedge=\Ucal^o\cap\Ucal^\perp$ is the essential image
    of the restriction functor $\lambda_\ast:\Mla_\Ucal\to\Mla$,
  \item $\Lambda_\Ucal/\Lambda$ is a directed union of finite
    extensions of modules in $\Ucal$,
  \item $T_\Ucal=\Lambda_\Ucal\oplus \Lambda_\Ucal/\Lambda$ is a
    tilting module with tilting class $\Gen T_\Ucal=\Ucal^\perp$.
\end{enumerate}
\end{thm}
\begin{proof}
  Let $\mathcal E$ be the extension closure of $\mathcal U$.  First of
  all, note that $\Lambda_{\mathcal U}$ coincides with
  $\Lambda_\Ecal$, and $\Ucal^o=\Ecal^o$, $\Ucal^\perp=\Ecal^\perp$,
  $\Ucal^\wedge=\Ecal^\wedge$, cf.~\cite[1.7]{AS2}.  Further, $\Ecal$
  is a class of finitely presented modules of projective dimension one
  which is closed under images, kernels, cokernels, and extensions,
  such $\Lambda\in\Ecal^o$, so it is a well-placed subcategory of
  bound modules in the terminology of \cite{S1,S2}. It then follows
  from \cite[5.5 and 5.7]{S1} that $\lambda$ is an injective
  homological epimorphism. Statement (2) is shown in
  \cite[2.7]{angarc}. Moreover, since $\Lambda$ is noetherian, for any
  finitely generated module $M$, the torsion submodule of $M$ with
  respect to the torsion pair generated by $\Ecal$ is finitely
  generated. Then one shows as in \cite[2.6]{S2} that $\Lambda_\Ecal$
  is a directed union of modules $M_t$ containing $\Lambda$ such that
  $M_t/\Lambda\in\Ecal$, and statement (3) is an immediate
  consequence. Finally, since $\Lambda$ is perfect, the class of
  modules of projective dimension at most one is closed under direct
  limits. We infer that $\Lambda_\Ucal/\Lambda$ and $\Lambda_\Ucal$
  have projective dimension at most one, and \cite[3.10 and
  4.12]{ringepi} yield statement (4).
\end{proof}

Let us describe the left adjoint
$\lambda^\ast=-\otimes_\Lambda\Lambda_\Ucal:\Mla\to\Mla_\Ucal$ of the
restriction functor $\lambda_\ast:\Mla_\Ucal\to\Mla$.

\begin{lemma}\label{leftadj}(cf.~\cite[1.7]{AS2}) Let $\mathcal U $ be
  a set of simple regular modules,  
  let $(\Tcal, \Fcal)$ be the torsion pair generated by $\mathcal U$,
  and let $t$ be the associated torsion radical.
\begin{enumerate}
\item $\Tcal=\{X\in\M\,\mid\, X\otimes_\Lambda \Lambda_\Ucal=0\}$.
\item Every $A\in\Mla$ admits a short exact sequence
  $$0\to A/tA\to A\otimes_\Lambda \Lambda_\Ucal\to A\otimes_\Lambda
  \Lambda_\Ucal/\Lambda\to 0$$
  where $A\otimes_\Lambda \Lambda_\Ucal\in\mathcal U^\wedge$ and
  $A\otimes_\Lambda \Lambda_\Ucal/\Lambda\in\Tcal$.
\end{enumerate} 
\end{lemma}

\begin{prop}\label{prop:sumofprufer} 
  Let $\Ucal$ be a set of simple regular modules.
\begin{enumerate}
\item $\Lambda_\Ucal$ is a torsion-free, and
  $\Lambda_\mathcal{U}/\Lambda$ is a torsion regular
  $\Lambda$-module. If $\mathcal{U}$ is a union of cliques, then
  $\Lambda_\mathcal{U}/\Lambda$ is a direct sum of all Pr\"ufer
  modules from the corresponding tubes.
\item If $\mathcal{U}$ does not contain a complete clique, then
  $\Lambda_\Ucal$ is a concealed canonical algebra with canonical
  trisection $(\p_\Ucal,\tube_\Ucal, \q_\Ucal)$, and the functors
  $\lambda_\ast$ and $\lambda^\ast$ map $\add\p_\Ucal$ to $\add\p$,
  $\add\tube_\Ucal$ to $\add\tube$, $\add\q_\Ucal$ to $\add\q$, and
  viceversa.  In particular,
\begin{enumerate}
\item the simple regular $\Lambda_\Ucal$-modules are precisely the
  modules of the form $S\otimes_\Lambda \Lambda_\Ucal$ where
  $S\not\in\Ucal$ is simple regular;
\item the Pr\"ufer modules over $\Lambda_\Ucal$ are precisely the
  modules of the form
  $S_\infty\otimes_\Lambda \Lambda_\Ucal\cong S_\infty$ where
  $S\not\in\Ucal$ is simple regular;
\item every $A\in\Ucal^o$ admits a short exact sequence
  $0\to A\to A\otimes_\Lambda \Lambda_\Ucal\to A\otimes_\Lambda
  \Lambda_\Ucal/\Lambda\to 0,$
  where $A\otimes_\Lambda \Lambda_\Ucal/\Lambda$ has a finite
  filtration by modules in $\Add\,\Ucal$, and thus lies in
  ${}^\perp(\Ucal^\perp)$;
\item   $\mathbf L\otimes_\Lambda \Lambda_\Ucal$ is the Lukas tilting module over $ \Lambda_\Ucal$.
\end{enumerate}
 \end{enumerate}
\end{prop}
\begin{proof} The first part of (2) is shown as in \cite[Proposition
  4.2 (Going down)]{LP}, while (1) and (2)(b) are proven as in
  \cite[Propositions 1.8, 1.10, and 1.11]{AS2}.

  In order to prove 2(c), we assume w.l.o.g. that $\Ucal$ consists of
  $m<r$ simple regular modules from a tube of rank $r$, and we proceed
  by induction on $m$.

  For $m=1$ we have $\Ucal=\{S\}$ for a simple regular in a tube of
  rank $r>1$, and $\lambda_\ast$ is the embedding of the perpendicular
  category $S^\wedge$ of $S$.  By the construction of the left adjoint
  $\lambda^\ast$ in \cite[1.3]{CTT} we know that the short exact
  sequence in Lemma \ref{leftadj} has the form
  $0\to A\to A_0\to S^{(c)}\to 0$ where $c$ is the minimal number of
  generators of $\Ext^1_\Lambda (S, A)$ as a module over
  $\End_\Lambda S$.

  Let now $1<m<r$, and choose a numbering $\Ucal=\{S_1,\ldots,S_m\}$
  such that $\Ext^1_\Lambda(S_i,S_m)=0$ for all $1\le i<m$.  Then
  taking $\Ucal'=\{S_1,\ldots,S_{m-1}\}$, we have that
  $S_m\in(\Ucal')^\wedge$ is a regular $\Lambda_{\Ucal'}$-module. So
  $\Lambda_\Ucal\cong(\Lambda_{\Ucal'})_{S_m}$ by \cite[4.6]{Scho},
  and we can compute
  $A\otimes_\Lambda \Lambda_{\Ucal}\cong (A\otimes_\Lambda
  \Lambda_{\Ucal'})\otimes_{\Lambda_{\Ucal'}}(\Lambda_{\Ucal'})_{S_m}$.
  By induction assumption we have a short exact sequence
  \[\xymatrix{0\ar[r] &A \ar[r]^{}&A\otimes_\Lambda
    \Lambda_{\Ucal'}\ar[r]^{}&A\otimes_\Lambda
    \Lambda_{\Ucal'}/\Lambda\ar[r]{}&0}
 \]
 where $A\otimes_\Lambda \Lambda_{\Ucal'}/\Lambda$ is a finite
 extension of modules in $\Add\,\Ucal'$, and in particular, it belongs
 to $\{S_m\}^o$. Then also
 $A\otimes_\Lambda \Lambda_{\Ucal'}\in \{S_m\}^o$.  Arguing as in the
 case $m=1$, we get an exact sequence
 $0\to A\otimes_\Lambda \Lambda_{\Ucal'}\to A\otimes_\Lambda
 \Lambda_{\Ucal}\to S_m\,^{(c)}\to 0$
 together with a commutative diagram
 \[\xymatrix{ &&0\ar[d]^{}&0\ar[d]^{}&
   \\
   0\ar[r] &A \ar[d]^{=}\ar[r]^{}&A\otimes_\Lambda
   \Lambda_{\Ucal'}\ar[d]^{}\ar[r]^{}&A\otimes_\Lambda
   \Lambda_{\Ucal'}/\Lambda\ar[d]^{}\ar[r]{}&0
   \\
   0\ar[r] &A \ar[r]^{}&A\otimes_\Lambda
   \Lambda_{\Ucal}\ar[d]^{}\ar[r]^{}&A\otimes_\Lambda
   \Lambda_{\Ucal}/\Lambda\ar[d]\ar[r]{}&0
   \\
   & &S_m\,^{(c)}\ar[d]^{}\ar[r]^{=}&S_m\,^{(c)}\ar[d]^{}&
   \\
   &&0&0& }
 \]
which yields the claim.

In 2(d), we specialize to $A=\mathbf L$, which certainly belongs to
$\Ucal^o$ as it is torsion-free. We get a short exact
sequence
$$0\to\mathbf L\to\mathbf L\otimes_\Lambda \Lambda_\Ucal\to \mathbf
L\otimes_\Lambda \Lambda_\Ucal/\Lambda\to 0$$
where the two outer terms have projective dimension at most one over
$\Lambda$, hence so does $\mathbf L\otimes_\Lambda
\Lambda_\Ucal$.
Since $\lambda$ is a homological epimorphism, it follows that
$\mathbf L\otimes_\Lambda \Lambda_\Ucal$ is a $\Lambda_\Ucal$-module
of projective dimension at most one. The remaining part of the proof
works as in \cite[Theorem 6]{baerml}.
\end{proof}

\section{Gabriel localizations of the heart}\label{Gabriel loc}
Aim of this section is to investigate the Gabriel localizations of
$\A=\Qcoh\XX$.  This leads to classification results for tilting or
cotilting modules over $\Lambda$.  More precisely, we are going to
classify the tilting and cotilting modules in the
class $$\Mcal=\Bcal\cap\Ccal.$$ Observe that $\add\tube$ is the class
of finite dimensional modules in $\Mcal$, so there are no finite
dimensional tilting modules in $\Mcal$. Indeed, given a tilting module
$T$, the number of pairwise non-isomorphic indecomposable summands
from $\tube$ is bounded by $\sum_{i=1}^t (p_i-1)$, where
$p_1,\ldots,p_t$ are the ranks of the non-homogeneous tubes in
$\tube$, and it is therefore strictly smaller than the rank of the
Grothendieck group ${\rm rk}\,K_0({\Lambda})=\sum_{i=1}^t (p_i-1)+2$.

\smallskip

On the other hand, the modules in $\add\tube$ can occur as direct
summands of a tilting module. Here is a first structure result.

\begin{prop}\label{shapeofT} Let $T$ be a tilting module in $\Ccal$.  Every module  $X\in\Add T$  has a unique 
  decomposition $ X=X'\oplus \overline{X}$ where $\overline{X}$ is
  torsion-free and $X'$ is a direct sum of Pr\"ufer modules and
  modules from $\tube$.
\end{prop}
\begin{proof}
  The proof of \cite[Proposition 4.2]{AS2}, is still valid in our
  context. We only have to explain why the torsion part $X'$ of a
  module $X\in\Add T$ is again in ${}^\perp(T^\perp)$. This can be
  seen by applying the functor $\Hom_\Lambda(-,B)$ with $B\in T^\perp$
  on the canonical sequence $0\to X'\to X\to \overline{X}\to 0$,
  keeping in mind that $\overline{X}$ is in $\Ccal$ and thus has
  projective dimension at most one by Lemma \ref{5.4}.
\end{proof}

One shows as in \cite[\S 3]{AS2} that the direct summands from $\tube$
in a tilting module $T\in\Mcal$ are arranged in disjoint wings, and
the
direct sum $Y$ of a complete irredundant set of such summands  is  a module of the following form.
 
 \smallskip
 
{\bf Definition}\label{branch}
A multiplicity free $\Lambda$-module $Y\in\add\tube$ is called a {\em
  branch module} if it satisfies
\begin{enumerate}
\item[(B1)] $\Ext^1_\Lambda({Y,Y})=0$,
\item[(B2)] for each simple regular module $S$ and $m\in\N$ such that $S_m$
is a direct summand of $Y$, there exist precisely $m$ direct summands of $Y$
that belong to $\mathcal{W}_{S_m}$.
\end{enumerate}

\smallskip

Finite dimensional torsion summands can be ``removed'' by employing
universal localization.

\begin{prop}\label{downtoLukas}
  Let $T=Y\oplus M$ be a tilting module where $0\not=Y\in\add\tube$,
  and let $\Ucal$ be the set of regular composition factors of
  $Y$. Assume that $M\in\Ucal^o$.  Then $M$ is a tilting module over
  $\Lambda_\Ucal$, which is large if and only if $T$ is large. In
  particular, if $M$ is a torsion-free $\Lambda$-module, then it is a
  torsion-free tilting module over $\Lambda_\Ucal$.
\end{prop}
\begin{proof} 
  In order to show that $M$ is a $\Lambda_\Ucal$-module, we have to
  verify $M\in\Ucal^\perp$. This is deduced inductively from the fact
  that $M\in Y^\perp$. In fact, if we assume w.l.o.g. that $Y=S_m$ is
  indecomposable, then $\Ext^1_\Lambda(S_i,M)=0$ for all its regular
  submodules $S_i, i<m$, because $Y/S_i$ has projective dimension
  one. Then, applying $\Hom(-,M)$ on the Auslander-Reiten sequences
  and keeping in mind that $\Hom(S_i, M)=0$, one obtains
  $\Ext^1_\Lambda((\tau^-S)_i,M)=0$ for all $i<m$.

  Now, since $\lambda:\Lambda{\to} \Lambda_\Ucal$ is a homological
  ring epimorphism, a $\Lambda_\Ucal$-module $N$ satisfies
  $\Ext^i_{\Lambda_\Ucal}(M,N)=0$ for some $i\ge 0$ if and only if
  $\Ext^i_{\Lambda}(T,N)=0$. It follows immediately that $M$ fulfills
  conditions (T1), (T2), and (T3) over $\Lambda_\Ucal$.

  As discussed above, $\Ucal$ cannot contain a complete clique, so
  $\Lambda_\Ucal$ is a concealed canonical algebra. From
  $\Lambda_\Ucal\in\mla$ we infer $\mla_\Ucal\subset\mla$, which shows
  that $M$ is equivalent to a finite dimensional tilting module over
  $\Lambda_\Ucal$ if and only if so is $T$ over $\Lambda$.

  In the special case when $M$ is torsion-free, the assumption
  $M\in \Ucal^o$ is satisfied. Further, by Proposition
  \ref{prop:sumofprufer}, the simple regular $\Lambda_\Ucal$-modules
  are precisely the modules of the form
  $S\otimes_\Lambda \Lambda_\Ucal$ where $S\not\in\Ucal$ is simple
  regular, and for any such $S$ there is an exact sequence
  $0\to S\to S\otimes_\Lambda \Lambda_\Ucal\to S\otimes_\Lambda
  \Lambda_\Ucal/\Lambda\to 0,$
  where $S\otimes_\Lambda \Lambda_\Ucal/\Lambda$ is a finite extension
  of modules in $\Add\,\Ucal$. Applying $\Hom_\Lambda(-,M)$ we get
  $0= \Hom_\Lambda(S\otimes_\Lambda \Lambda_\Ucal/\Lambda, M)\to
  \Hom_\Lambda(S\otimes_\Lambda \Lambda_\Ucal, M) \to \Hom_\Lambda(S,
  M)=0$, showing that $M$ is torsion-free over $\Lambda_\Ucal$.
\end{proof}

\medskip

\subsection{\sc Tilting modules in $\Mcal$.}

Let now $Y$ be a branch module, and let $\Ucal$ be the set of all
regular composition factors of $Y$. Denote
    $$T_{(Y,\emptyset)}=Y\oplus (\mathbf{L}\otimes_\Lambda \Lambda_{\mathcal U}).$$
    Moreover, given a non-empty subset $P\subset \mathbb X$, let
    $\mathcal{V}$ be the union of $\Ucal$ with the cliques of the
    tubes $\Ucal_x, x\in P$, and set
    $$T_{(Y,P)}=Y\oplus \bigoplus\{\text{all } S_\infty \text{ in }
    {}^\perp Y \text{ from tubes } \Ucal_x, {x\in P}\}\oplus
    \Lambda_{\mathcal V}$$ Dually, consider
$$C_{(Y,P)}=Y\,\oplus \prod\{\text{all } S_{-\infty} \text{ in }
Y^\perp\text{ from }  \Ucal_x, {x\in P} \}\oplus G\oplus
\bigoplus\{\text{all } S_\infty \text{ in } {}^\perp Y \text{ from } 
\Ucal_x, {x\not\in P}\}$$
Notice that by the Auslander-Reiten formula, $S_{-\infty} \in Y^\perp$
\ifa\ $S$ does not occur as a regular composition factor of $\tau Y$,
and similarly, $S_\infty \in {}^\perp Y$ \ifa\ $S$ does not occur as a
regular composition factor of $\tau^- Y$.
 
 \smallskip
 
 We will prove that the $T_{(Y,P)}$ and the $C_{(Y,P)}$ give a
 complete list of all tilting, respectively cotilting, modules in
 $\Mcal$, up to equivalence. We start by collecting some information
 on the class $\Mcal$.
 
  \begin{thm}\label{classpi}  \cite[2.2]{Rac}
    A $\Lambda$-module $M$ is pure-injective and belongs to $\Mcal$ if
    and only if there is a decomposition $M=M'\oplus M''$ where
    $M'\in\Prod\tube$ and $M''\in\Add\mathbf W$.
 \end{thm}
 
 \begin{lemma}\label{SandT}
   Let $T$ be a tilting module, and let
   $\Scal={}^\perp(T^\perp)\cap\mla$ be the resolving subcategory
   corresponding to $T$ under the bijection of Theorem
   \ref{res}. \begin{enumerate}
\item $T\in\Ccal$ if and only if $\Scal\subset\Ccal$.
\item  $T\in\Bcal$ if and only if $\p\subset\Scal$.
\item $T\in\Mcal$ if and only if $\Scal=\add(\p\cup\tube')$ for some
  subset $\tube'\subset\tube$.  In particular, $\Scal$ is then closed
  under submodules.
\item Assume $T\in\Ccal$. Then the ray of a simple regular module $S$
  is completely contained in $\Scal$ if and only if the corresponding
  Pr\"ufer module $S_\infty$ is a direct summand of $T$. On the other
  hand, if $\Scal$ contains some, but not all modules from that ray,
  then $T$ has a direct summand $S_m\in\Scal$ (which is the module of
  maximal regular length in $\tube'\cap\{S_n\mid n\in\N\}$).
\end{enumerate} 
\end{lemma}
\begin{proof}
  (1) First of all, recall that the cotilting class
  $\Ccal={}^\perp(\Ccal^\perp)$ is closed under direct summands and
  filtrations, cf.~\cite[3.1.2]{GT}. The if-part then follows from the
  fact that $T$ is a direct summand of an $\Scal$-filtered module by
  \cite[3.2.4]{GT}. Conversely, $T\in\mathcal C$ implies
  $\C^\perp\subset T^\perp$, and
  ${}^\perp(T^\perp)\subset{}^\perp(\C^\perp)=\C$, hence
  $\Scal\subset \C$.

  (2) Since $\Bcal$ is a torsion class, $T\in\Bcal$ if and only if
  $T^\perp=\Gen T\subset \Bcal=\p^\perp$, that is,
  ${}^\perp(\p^\perp)\subset{}^\perp(T^\perp)$, and as above we see
  that the latter is equivalent to $\p\subset \Scal$.

  (3) The first statement follows combining (1) and (2).  Now, if
  $A'\in\add\tube$ is a submodule of $A\in\add\tube'$, then
  $A/A'\in\add\tube$ has projective dimension one, hence
  $A'\in{}^\perp(T^\perp)$ if so does $A$. This shows that $\Scal$ is
  closed under submodules.

(4)  is shown as  in \cite[4.5 and 3.3]{AS2}.
\end{proof}

 Now we are ready for our classification result.
 
\begin{thm}\label{tiltinginM}
   There are  one-one-correspondences between
\begin{enumerate}
\item[(1)] 
the pairs  $(Y,P)$ where $Y$ is a branch module and $P\subset\mathbb X,$
\item[(2)]
the equivalence  classes of  tilting modules in $\Mcal$,
\item[(3)] the equivalence  classes of cotilting modules in $\Mcal$,
\end{enumerate}
assigning to $(Y,P)$ the (equivalence class of) the tilting module
$T_{(Y,P)}$ and the cotilting module $C_{(Y,P)}$ defined above.
\end{thm}
\begin{proof}
  The proof of the bijection between (1) and (2) works as in
  \cite[Theorem 5.6]{AS2}, we give an outline and point out the
  arguments that have to be modified.

  Given a tilting module $T$ in $\Mcal$, let $Y$ be the branch module
  obtained as direct sum of a complete irredundant set of finite
  dimensional indecomposable summands of $T$, and let $\mathcal U$ be
  the set of regular composition factors of $Y$.

  According to Lemma \ref{SandT}, there is a subset
  $\tube'\subset\tube$ such that
  $\Scal={}^\perp(T^\perp)\cap\mla=\add(\p\cup\tube')$.  Moreover, the
  ray of a simple regular module $S$ is completely contained in
  $\tube'$ if and only if the corresponding Pr\"ufer module $S_\infty$
  is a direct summand of $T$.

\smallskip

{\it The first part of the proof}  now consists in showing that 
\begin{enumerate}
\item[(i)] If $\tube'$ contains no complete ray from a tube in
  $\tube$, then $T$ is equivalent to a tilting module of the form
  $T_{(Y,\emptyset)}=Y\oplus ({\mathbf L}\otimes
  \Lambda_\mathcal{U})$.
\item[(ii)] If $\tube'$ contains some rays, then $T$ is equivalent to
  a tilting module of the form $T_{(Y,P)}$ where
  $P=\{x\in\mathbb X\mid \tube'\text{ has rays from }\Ucal_x\}$.
\end{enumerate}
More precisely, one shows that $T$ is equivalent to a tilting module
of the form $T'=Y\oplus M$, where $M$ is a tilting module over the
universal localization at a suitable set of simple regular modules
$\Ucal'$ that contains no complete clique, and further, $M$ is chosen
such that the corresponding resolving subcategory of $\mla_{\Ucal'}$
is the localization
$\Scal\otimes\Lambda_{\Ucal'}=\{A\otimes_\Lambda \Lambda_{\Ucal'}\mid
A\in\Scal\}$ of $\Scal$.

In case (i), $\tube'$ is contained in the extension closure of $\Ucal$
by Lemma \ref{SandT}. We take $\Ucal'=\Ucal$, hence
$\Scal\otimes\Lambda_{\Ucal'}=\add\p_\Ucal$, and $M$ is the Lukas
module over $\Lambda_\Ucal$, that is
$M=\mathbf L\otimes_\Lambda\Lambda_\Ucal$ by Proposition
\ref{prop:sumofprufer}. So $T'=T_{(Y,\emptyset)}$.

In case (ii), $\tube'$ consists of the rays it contains (from tubes
$\Ucal_x, x\in P$) and of a subset of the extension closure of
$\Ucal$. Take $\mathcal U'$ as follows:
$\Ucal'\cap\Ucal_x=\Ucal\cap\Ucal_x$ if $x\not\in P$, while if
$x\in P$, then $\Ucal'\cap\Ucal_x$ consists of the simple regular
modules whose ray is not completely contained in $\tube'$.
Further, let $\mathcal{V}$ be the union of $\Ucal'$ (or equivalently,
of $\Ucal$) with the cliques of the tubes $\Ucal_x, x\in P$. Then the
localization $\tube'\otimes\Lambda_{\Ucal'}$ of $\tube'$ at $\Ucal'$
coincides with the localization of the tubes $\Ucal_x,x\in P$, and it
is given by the set of simple regular $\Lambda_{\Ucal'}$-modules
$\mathcal V'=\{S\otimes \Lambda_{\Ucal'}\,\mid\,S\in\mathcal
V\setminus\Ucal'\}$
corresponding to the simple regular $\Lambda$-modules $S$ whose ray is
contained in $\tube'$. This shows that $\Scal\otimes\Lambda_{\Ucal'}$
is the additive closure of $\p_{\Ucal'}$ and a union of tubes, and the
corresponding tilting class is
$\{X\in\Mla_{\Ucal'}\mid\Ext^1_{\Lambda_{\Ucal'}}(V',X)=0\text{ for
  all } V'\in\mathcal V'\}$.
Applying Theorem \ref{localizingwide} on the canonical algebra
$\Lambda_{\Ucal'}$, and keeping in mind that
$(\Lambda_{\Ucal'})_{\mathcal V'}\cong \Lambda_{\mathcal V}$ by
\cite[4.6]{Scho}, we conclude that we can take
$M=\Lambda_{\mathcal V}\oplus \Lambda_\mathcal V/\Lambda_{\Ucal'}$.
Furthermore one infers from Proposition \ref{prop:sumofprufer} that
$\Lambda_\mathcal V/\Lambda_{\Ucal'}\cong\bigoplus\{S_\infty\mid
S\in\mathcal V\setminus\mathcal U'\}$,
which by definition of $\mathcal V$ and $\Ucal'$ is isomorphic to
$\bigoplus\{\text{all } S_\infty \text{ in } {}^\perp Y \text{ from
  tubes } \Ucal_x, {x\in P}\}$. Hence $T'= T_{(Y,P)}$.
  
In order to prove that $T$ is actually equivalent to the tilting
module $T'=Y\oplus M$ with $M$ as explained above, one proceeds as in
the proof of \cite[Propositions 5.4 and 5.5]{AS2} by verifying
  \begin{enumerate}
  \item[(i)] if $\mathcal W$ is the extension closure of $\Ucal'$,
    then $\mathcal W\cup\mathcal W^o$ contains $\tube'$,
 \item[(ii)] $\Add(\tube\cap\Gen T)\subset M^\perp$,
 \item[(iii)] every torsion-free module in $\Prod\tube\cap\Gen T$ is
   contained in $\mathcal U'\,^\perp$.
\end{enumerate}   
These conditions are an adapted version of \cite[Proposition
5.2]{AS2}, where (iii) has been modified in view of the classification
of pure-injectives in Theorem \ref{classpi}.  The proof of
\cite[Proposition 5.2]{AS2} works also in this context. In fact, we
only have to change the argument for checking that the module $M$ is
in $T^\perp$ on \cite[page 31, lines 23-26]{AS2}: in order to verify
that every $A\in\Scal\cap\Ucal ^o$ belongs to ${}^\perp M$, we use the
exact sequence
$0\to A\to A\otimes_\Lambda \Lambda_\Ucal\to A\otimes_\Lambda
\Lambda_\Ucal/\Lambda\to 0$
from Proposition \ref{prop:sumofprufer} where
$A\otimes_\Lambda \Lambda_\Ucal/\Lambda$ has projective dimension
one. Then $A\in{}^\perp M$ whenever
$A\otimes_\Lambda \Lambda_\Ucal\in{}^\perp M$, and the latter holds
true by choice of $M$.  Notice that \cite[Proposition 5.2]{AS2} relies
on \cite[Lemma 5.1]{AS2}. For proving the latter in our context, it
remains to explain why condition (T1) is verified: this follows from
the assumption $T^\perp\subset M^\perp$, which yields that $M$ belongs
to the class $ {}^\perp( T^\perp)$ consisting of modules of projective
dimension at most one.
 
 \smallskip
 
 {\it The second part of the proof} is devoted to establishing the
 stated bijection. Given a pair $(Y,P)$ as in the Theorem, one
 proceeds as in the proof of \cite[Theorem 5.6]{AS2} to construct a
 tilting module $T\in\Mcal$ which, according to the first part of the
 proof, must be equivalent to $T_{(Y,P)}$. So the assignment
 $(Y,P)\mapsto T_{(Y,P)}$ is well defined and surjective. For the
 injectivity, one uses Proposition \ref{shapeofT} to see that the
 equivalence class of $T_{(Y,P)}$ determines the torsion part of the
 tilting module and thus the pair $(Y,P)$.

\smallskip

{\it The last part of the proof} is devoted to the bijection between
(1) and (3). First of all, since the dual of a branch right module $Y$
is a branch left module $Y'=D(Y)$, we also have a bijection between
the pairs in (1) and the equivalence classes of left tilting
$\Lambda$-modules in the corresponding class $_\Lambda\Mcal$ in
$\LMla$.  But we know from Theorem \ref{res} that the duality $D$
yields a bijection between left tilting and right cotilting
$\Lambda$-modules, and it is clear that $T\in{}_\Lambda\Mcal$ if and
only if $D(T)$ is in $\Mcal$. So, we obtain also a bijection between
(1) and (3).

It only remains to verify that the duals of the left tilting modules
$T_{(Y',P)}$ are of the form $C_{(Y,P)}$ with $Y=D(Y')$, up to
equivalence. Certainly, $C=D(T_{(Y',P)})$ is a cotilting right module
isomorphic to
$$Y\oplus \prod\{\text{all } S_{-\infty} \text{ in } Y^\perp\text{
  from tubes } \Ucal_x, {x\in P} \}\oplus D(M)$$
where $M\in{}_\Lambda\Mcal$ is the torsion-free part of
$T_{(Y',P)}$. Then $D(M)\in\Mcal$ is divisible and pure-injective, and
by Theorem \ref{classpi} it is a direct sum of Pr\"ufer modules and
copies of $G$.  Of course, the Pr\"ufer modules occurring in $\Prod C$
must lie in ${}^\perp Y$. Moreover, they must belong to tubes
$\Ucal_x,x\not\in P$, because one can show as in \cite[2.7]{BK} that
$\Ext^1_\Lambda(S_\infty,S'_{-\infty})=0$ if and only if and $S$ and
$S'$ do not belong to the same clique.

Conversely, any such Pr\"ufer module $S_\infty$ is Ext-orthogonal to
$Y$, to all adic modules $S_{-\infty}$ from tubes $\Ucal_x, {x\in P},$
to $G$ and to any Pr\"ufer module, and therefore it lies in
${}^\perp C=\Cogen C$. We claim that it even belongs to $\Prod C$.

In fact, since $C\in\C={}^\perp\q$, we have $\Cogen C\subset\C$ and
therefore $\Cogen C\cap\mla\subset\add(\p\cup\tube)$. As in the proof
of \cite[Theorem A.1]{AS2}, we infer that all pure-injective divisible
modules, hence in particular $G$ and the Pr\"ufer modules, belong to
$({}^\perp C)^\perp$. So the Pr\"ufer modules $S_\infty\in\Cogen C$
belong to ${}^\perp C\cap ({}^\perp C)^\perp=\Prod C$. As for the
generic module, recall that $G$ is torsion-free and thus
$G\in\varinjlim\p$. Then using that $C\in\Bcal=\p^\perp$ and that
${}^\perp C$ is closed under direct limits, we deduce that
$G\in{}^\perp C$ and therefore also $G\in\Prod\,C$.

So, we conclude that, up to multiplicities, $D(M)\cong G
\oplus \bigoplus\{\text{all } S_\infty \text{ in } {}^\perp Y \text{ from tubes }
    \Ucal_x, {x\not\in P}\}$ and thus $C$ is equivalent to $C_{(Y,P)}$.
\end{proof}

\medskip

\subsection{\sc Gabriel localizations of  $\Acal$.}
We now turn to the relationship with Gabriel localization. Recall that
$\Acal=\Qcoh\XX$ is a locally noetherian hereditary Grothendieck
category, with
${\rm Inj}{\mathcal A}=\{G[1]\}\cup\{S_\infty[1]\mid S\text{ simple
  regular}\}$
being the set of indecomposable injective objects in $\mathcal A$. We
consider the {\it Gabriel topology} on ${\rm Inj}{\mathcal A}$ with
closed sets
$$I(\Serre)=\{E\in {\rm Inj}{\mathcal A}\,\mid\,
\Hom_\Acal(C,E)=0\;\text{for\ all}\;C\in\Serre\}$$
where $\Serre$ runs through the {Serre subcategories} of $\fpA$. The
torsion pair $(\Tcal,\Fcal)$ cogenerated by $I(\Serre)$ coincides with
the one generated by $\Serre$ and is a {hereditary} torsion pair with
$\Tcal=\varinjlim \Serre$. The assignments $\Serre\mapsto I(\Serre)$,
and $\Serre\mapsto(\Tcal,\Fcal)$ are part of the following
correspondence.
 
\begin{thm}(\cite{Ga},\cite[2.8 and 3.8]{He},\cite{K},\cite[Ch. 11]{P}) \label{Gabriel}
There is a bijection between 
\begin{enumerate}\item 
the hereditary torsion pairs  in $\mathcal A$,
\item the closed subsets of ${\rm Inj}{\mathcal A}$,\item 
and the Serre subcategories of $\fpA$,
\end{enumerate}
which assigns  to a hereditary torsion pair $(\Tcal, \Fcal)$ 
the closed subset $\Ical={\rm Inj}{\mathcal A}\cap\Fcal$ and 
the Serre subcategory $\Serre={\rm fp}{\Acal}\cap\Tcal$.
\end{thm}

If a Serre subcategory $\Serre\subset\fpA$ contains an indecomposable
vector bundle, then it contains all vector bundles, and therefore the
corresponding torsion theory $(\Tcal, \Fcal)=(\Acal, 0)$
is trivial, cf.~\cite[9.2]{GL}. So the hereditary torsion pairs
$(\T,\F)$ in $\mathcal A$ with non-trivial $\Fcal$ correspond to the
Serre subcategories consisting of finite length objects. Then
$\text{vect}\mathbb X\subset\F$, or equivalently,
$\T\subset{}^o(\text{vect}\mathbb X)$, and $(\Tcal,\Fcal)$ is faithful
since $\text{vect}\mathbb X$ contains a set of generators of $\Acal$.

Recall that the category of finite length objects in $\Acal$ is given
by the tubular family
$\mathcal H_0=\tube[1]=\bigcup_{x\in\mathbb X} \Ucal_x[1].$
One easily verifies that the Serre subcategories of $\add\mathcal H_0$
are precisely the small additive closures of unions of tubes and wings
in $\tube[1]$.

\medskip

\begin{cor}\label{onlysur}
  There is a surjective map from the set of equivalence classes of
  cotilting modules in $\Mcal$ to the set of faithful hereditary
  torsion pairs in $\A$. It assigns to the (equivalence class of the)
  cotilting module $C_{(Y,P)}$ the torsion pair $(\Tcal,\Fcal)$ in
  $\A$ cogenerated by the indecomposable injective summands of
  $C_{(Y,P)}[1]$.
\end{cor}
\begin{proof}
  Let
  $I=\{G[1]\}\cup\{S_\infty[1]\mid S_\infty\text{ in } {}^\perp Y
  \text{ from } \Ucal_x, {x\not\in P}\}\}\subset{\rm Inj}{\mathcal A}$
  be the set of indecomposable injective summands of
  $C_{(Y,P)}[1]$. It is easy to see that $I=I(\Serre)$ for the Serre
  subcategory
  $\Serre =\Serre_{(Y,P)} =\add(\bigcup_{x\in P}
  \Ucal_x[1]\cup\bigcup_{i=1}^{r} \mathcal W_i[1])$
  where $\mathcal W_1,\ldots, \mathcal W_r$ are the wings defined by
  the regular composition factors of $\tau^{-}Y$.  Since all Serre
  subcategories of $\add\mathcal H_0$ have this form, the assignment
  $(Y,P)\mapsto\Serre_{(Y,P)}$ is surjective. Now Theorems
  \ref{tiltinginM} and \ref{Gabriel} yield the statement.
 \end{proof}
 
 Clearly this map is not injective in general, because different
 branch modules can give rise to the same wings. Moreover, we point
 out that $(\Tcal,\Fcal)$ need not coincide with the torsion pair
 given by the torsion-free class $\Cogen C_{(Y,P)}[1]$. For example,
 if $\Ucal_x$ is a tube of rank 3, $S$ is a simple regular in
 $\Ucal_x$ and $S'=\tau^-S$, then $Y=S_2\oplus S'$ is a branch module
 and
 $C=C_{(Y,\emptyset)}=G\oplus Y\oplus S_\infty\oplus
 \bigoplus(\text{all Pr\"ufer modules from the other tubes})$
 is a cotilting module that gives rise to a non-hereditary torsion
 pair in $\A$. In fact, $\Cogen C[1]$ is not closed under injective
 envelopes since it does not contain the injective envelope
 $S'_\infty[1]$ of $S'[1]$.

\smallskip
 
In the homogeneous case, however, the parametrization in Theorem
\ref{tiltinginM} reduces to the subsets $P\subset\mathbb X$. So,
denoting by $\Lambda_P$ the universal localization at the cliques of
the tubes $\Ucal_x,x\in P,$ we have tilting modules
$T_P=\Lambda_P\oplus \Lambda_P/\Lambda$ when $P\not=\emptyset$, and
$T_\emptyset=\mathbf L$, as well as cotilting modules
$C_P=G\oplus \prod\{\text{all } S_{-\infty} \text{ from } \Ucal_x,
x\in P\}\oplus\bigoplus\{\text{all } S_\infty \text{ from } \Ucal_x,
x\not\in P\}$.
We then obtain a similar classification result as for commutative
noetherian rings, cf.~\cite{APST}.

\begin{cor}\label{Kronecker}
 Assume that $\Lambda$ is a tame bimodule algebra.  There is a bijection between 
\begin{enumerate}
\item the subsets of $\mathbb X$,
\item the equivalence  classes of tilting modules in $\Mcal$,
\item the equivalence  classes of cotilting modules in $\Mcal$,
\item the faithful hereditary torsion pairs in $\mathcal A$.
\end{enumerate}
The bijection assigns to a subset $P\subset\mathbb X$ the tilting
module $T_P$, the cotilting module $C_P$, and the faithful hereditary
torsion pair $(\T_P,\F_P)$ in $\A$ given by
$\T_P=\varinjlim(\bigcup_{x\in P} \Ucal_x[1])$ and
$\F_P=\Cogen C_P[1]$.  When $P\not=\emptyset$, the quotient category
$\A/\Tcal_P$ is equivalent to $\Mla_P$.
\end{cor}
\begin{proof}
  The bijection between the sets in (1) - (4) follows from Theorem
  \ref{tiltinginM} and the discussion above.  We only have to verify
  that the torsion pair $(\T_P,\F_P)$ in $\A$ cogenerated by the
  indecomposable injective summands of $C_P[1]$ has the stated shape.
  We have already seen in the proof of Corollary \ref{onlysur} that it
  coincides with the torsion pair generated by the Serre subcategory
  $\Serre_{P} =\add(\bigcup_{x\in P} \Ucal_x[1])$, so $\T_P$ looks as
  desired. Clearly, $\F_P\subset\Cogen C_P[1]$.  The reverse inclusion
  follows from the fact that $G[1]$ cogenerates all torsion-free
  objects in $\A$ by \cite[4.1]{RR}.

Further, it is well known (see e.g.\cite{K}) that
$$\A/\Tcal_P=\{X\in\A\mid\Hom_\A(Y,X)=\Ext^1_\A(Y,X)=0\text{ for all } Y\in\Serre_P\},$$
and by Proposition
\ref{localizingwide}
$$\Mla_P=\{X\in\Mla\mid\Hom_\Lambda(Y,X)=\Ext^1_\Lambda(Y,X)=0\text{
  for all } Y\in{\bigcup_{x\in P}} \Ucal_x\}.$$
By assumption, $\Qcal=\Add\q$ and no module from $\q$ can belong to
$(\bigcup_{x\in P} \Ucal_x)^o$, hence $\Mla_P\subset\Ccal$. Similarly,
$\A/\Tcal_P\subset\Ccal[1]$. So, the functor $H_V$ from Section
\ref{hearts of cotilting} yields the desired equivalence.
\end{proof}

We will see below that in the situation of Corollary \ref{Kronecker} a
(co)tilting module belongs to $\Mcal$ if and only if it is large.

\bigskip

\section{Large tilting and cotilting modules}\label{class results}
This Section aims at a classification of the large tilting and
cotilting modules for concealed canonical algebras of domestic or
tubular representation type.

We start out by recalling that cotilting modules over noetherian rings
are determined up to equivalence by their indecomposable summands,
which are known to be pure-injective by \cite{B1}.

\begin{thm}\cite{T}\label{superdec}
  Let $C$ be a cotilting left $R$-module over a left noetherian ring
  $R$. Then $\Prod C$ contains a family of indecomposable modules
  $(M_i)_{i\in I}$ such that
  $C$ is a direct summand in a direct limit of modules in
  $\Prod\{M_i\mid i\in I\}$ and
  ${}^\perp C=\bigcap_{i\in I} {}^\perp M_i$.
\end{thm}
\begin{proof} see \cite[Theorem 3.7]{T}  and its proof.\end{proof}
In the domestic case, the classification of large tilting modules has already been accomplished.
\begin{thm}\label{classdomestic}
There are  one-one-correspondences between
\begin{enumerate}
\item[(i)] 
the pairs  $(Y,P)$ where $Y$ is a branch module and $P\subset\mathbb X,$
\item[(ii)]
 the equivalence  classes of large tilting $\Lambda$-modules,
\item[(iii)] the equivalence  classes of large cotilting $\Lambda$-modules.
\end{enumerate}
\end{thm}
\begin{proof} The statement will follow from Theorem \ref{tiltinginM}
  once we prove that all large tilting or cotilting modules are in
  $\Mcal$. Let $T$ be a large tilting right module, and let $C=D(T)$
  be the dual cotilting left module. As in \cite[2.6]{AS2}, we infer
  that $C$ has to be large as well.

  Take a family of pairwise non-isomorphic indecomposable modules
  $(M_i)_{i\in I}$ in $\Prod C$ as in Theorem \ref{superdec}. The
  $M_i$ cannot be all finite dimensional. Indeed, otherwise the
  cardinality of the index set $I$ is bounded by the rank of the
  Grothendieck group, and so the module $M=\prod_{i\in I} M_i$ is
  finite dimensional, and $\Add M=\Prod M$ is definable. Since $C$ is
  a direct summand of a direct limit of modules in $\Prod M$, we infer
  that $C\in\Add M$ is equivalent to a finite dimensional cotilting
  module, a contradiction.

  So there is $i\in I$ such that $M_i$ is infinite dimensional, and by
  Proposition \ref{pi} it is isomorphic to a Pr\"ufer or adic module,
  or to the generic module. Now Lemma \ref{0} shows that no $M_i$ can
  belong to $\Qcal$ nor to $\Pcal$, hence all $M_i$ are in $\tube$, or
  Pr\"ufer, adic, or generic, and in particular they all belong to
  $\Mcal$. But $\Mcal$ is definable, so also $C$, as a direct summand
  of a direct limit of modules in $\Prod M$, belongs to $\Mcal$. And
  by duality $T$ has to belong to $\Mcal$ as well.
\end{proof}

\medskip

{\it From now on}, we assume that $\Lambda$ is concealed canonical of
tubular type.  The AR-quiver of $\Lambda$ then consists of a
preprojective component $\p_0$, a preinjective component $\q_\infty$
and a countable number of sincere separating tubular families
$\tube_\alpha,\alpha\in\Q_0^\infty=\Q^+\cup\{0,\infty\},$ where
$\tube_\alpha$ is stable precisely when $\alpha\in\Q^+$.  We fix $w\in
\R^+
$
and
set
$$\p_w=\p_0\cup\bigcup_{\alpha<w}\tube_\alpha,\qquad\q_w=\bigcup_{w<\gamma}\tube_\gamma\cup\q_\infty$$ 
For $w\in\Q^+$ we thus obtain a trisection $(\p_w,\tube_w,\q_w)$ as in
Section \ref{setup}, while for $w\in\R^+\setminus\Q^+$ the finitely
generated indecomposable modules all belong either to $\p_w$ or to
$\q_w$.

\smallskip

\subsection{\sc The slope of a module.} \label{slope}
Following \cite[\S 13]{RR}, we now consider torsion pairs in $\Mla$
constructed from the classes above and use them to define a notion of
slope.  We start with the torsion pair cogenerated by $\p_w$. By the
Auslander-Reiten formula $$\mathcal B_w = {}^o(\p_w) = (\p_w)^\perp$$
with $\add\p_w$ being resolving, so it follows from Theorem \ref{res}
 that $\mathcal B_w$ is  a tilting class. 
The corresponding  torsion--free class will be denoted by $\mathcal P_w$. 

We fix a tilting module $\mathbf L_w$ generating $\mathcal
B_w$.
Notice that it must be infinite dimensional: otherwise
$\mathbf L_w\in\mathcal B_w\cap{}^\perp(\mathcal B_w)\cap\mla=\mathcal
B_w\cap\add\p_w={}^o(\p_w)\cap\add\p_w$, which is impossible.

Furthermore, $ \mathcal B_w={}^o(\bigcup_{\alpha<w} \tube_\alpha).$
Indeed, if a module $X$ has a non-zero morphism $X\to P$ for some
$P\in\p_0$, then we consider the injective envelope $f:P\to
I(P)$.
Since all indecomposable injective modules lie in
$\tube_\infty\cup\q_\infty$, we can factor through any tube in any
tubular family $\tube_\alpha$ where $\alpha<w$. This shows that $P$
embeds in a module in $\add\tube_\alpha$, and so we obtain a non-zero
morphism $X\to Y$ with $Y\in\tube_\alpha$.

\smallskip

We now turn to the torsion pair generated by $\q_w$.
By the Auslander-Reiten formula
$$\mathcal C_w = {}(\q_w)^o={}^\perp (\q_w).$$ Since $\q_w$ is dual to
a class $_\Lambda\p_{\tilde w}\subset\lmla$, by the well-known Ext-Tor
formulae and Theorem \ref{res} it follows from
that $\mathcal C_w=(_\Lambda\p_{\tilde w})^\intercal$ is a cotilting
class given by a large cotilting module $\mathbf W_w$.  Dually,
$ \mathcal C_w=(\bigcup_{w<\gamma} \tube_\gamma)^o.$
The corresponding torsion class is  $\mathcal Q_w=\Gen\q_w$, and
 for $w\in\Q\cup\{0,\infty\}$ 
 the torsion pair $(\mathcal Q_w,\mathcal C_w)$ is split, see \cite[13.1]{RR}. 

\smallskip

Finally, let
$\p_\infty=\p_0\,\cup\,\bigcup_{\alpha<\infty}\tube_\alpha$ and
$\q_0=\bigcup_{0<\gamma}\tube_\gamma\,\cup\,\q_\infty$, and set as
above $\Bcal_w={}^o(\p_w)$, and $\C_w=(\q_w)^o$ for $w=0$ or
$w=\infty$.  According to \cite{RR}, the modules belonging to the
class $$\mathcal M_w=\mathcal C_w\cap\mathcal B_w$$ with
$w\in\R^+\cup\{0,\infty\}$ will be said to have {\it slope} $w$.

\smallskip

\subsection{\sc Rational slope}  \label{rational}
When $w\in\Q^+$, we are in the situation of Sections \ref{setup} and
\ref{ARF}.  So there is a tilting and cotilting module $\mathbf W_w$
which cogenerates $\mathcal C_w$ and generates the torsion class
$\mathcal D_w={}^o(\tube_w)=(\tube_w)^\perp$ with corresponding split
torsion pair $(\mathcal D_w,\mathcal R_w)$.  The module $\mathbf W_w$
can be chosen as the direct sum of a set of representatives of the
Pr\"ufer modules and the generic module $G_w$ from the family
$\tube_w$.

Furthermore,
$ \mathcal D_w={}^o(\bigcup_{\alpha\le w} \tube_\alpha)={}^o(\p_w\cup
\tube_w).$
Indeed, if a module $X$ has a non-zero morphism $X\to P$ for some
$P\in\p_w$, then as shown above we can assume that $P\in\tube_\alpha$
with $\alpha<w$. But since the modules in $\tube_\alpha$ are
cogenerated by $\tube_w$, we infer that there is a non-zero morphism
$X\to Y$ with $Y\in\tube_w$.

Similarly, the torsion pair $(\Gen \tube_w,\mathcal F_w)$ generated by
$\tube_w$ satisfies
$\mathcal F_w=(\bigcup_{w\le\gamma}
\tube_\gamma)^o=(\tube_w\cup\q_w)^o.$

\smallskip

We can now consider the t-structure induced by the torsion pair
$(\Gen\q_w,\mathcal C_w)$ in \Db. Its heart by $\mathcal A_w$ is
equivalent to the category $\Qcoh\XX_w$ of quasi-coherent sheaves over
a noncommutative curve of genus zero $\mathbb X_w$ parametrizing the
family $\tube_w$, which is again of tubular type by
\cite{LP},\cite[8.1.6]{K}.

According to Theorem \ref{tiltinginM}, the tilting and cotilting
modules of slope $w$ are then parametrized by the pairs $(Y,P)$ where
$Y$ is a branch object in $\tube_w$ and $P\subset\mathbb X_w$, and
they are related to the Gabriel localizations of $\Qcoh\XX_w$ as
explained in Corollary \ref{onlysur}.
 
\smallskip

\subsection{\sc Irrational slope} \label{irrational}
Let us first collect some properties of the classes introduced above. 

\begin{lemma} \label{BwCw}
Let $w\in\R^+$. Then \begin{enumerate}
\item
  $\mathcal B_w=\bigcap\limits_{\R^+\ni v< w}\mathcal
  B_v=\bigcap\limits_{\Q^+\ni \alpha< w}\mathcal D_\alpha$,
  \: and\:
  $\mathcal C_w=\bigcap\limits_{w<v\in\R^+}\mathcal
  C_v=\bigcap\limits_{w<\gamma\in\Q^+}\mathcal F_\gamma$.

\smallskip

\item  $\mathcal Q_w=\varinjlim \q_w$, and if $w\not\in\Q$, then $\mathcal C_w=\varinjlim \p_w$.

\smallskip

\item $\mathcal P_w\subset\mathcal C_w$ and
  $\mathcal Q_w\subset\mathcal B_w$.  If $w\in\Q^+$, then
  $\mathcal P_w\subset\mathcal F_w\subset\mathcal C_w$ and
  $\mathcal Q_w\subset\mathcal D_w\subset\mathcal B_w$.
 
 \smallskip
 
\item
  $(\mathcal C_w)^\perp\subset\mathcal B_w=\bigcap\limits_{\R^+\ni v<
    w}(\mathcal C_v)^\perp=\bigcap\limits_{\R^+\ni v< w}\mathcal Q_v$.
  
\end{enumerate}
 \end{lemma}
\begin{proof}
  (1) By definition, a module belongs to $\mathcal B_w$ if and only if
  it belongs to
  $ {}^o(\p_0\cup\bigcup_{\alpha<v}\tube_\alpha)=\mathcal B_v$ for all
  ${v< w}$. Moreover, by the description of $\mathcal B_w$ in Section
  \ref{slope}, we have
  $\mathcal
  B_w=\bigcap_{\alpha<w}{}^o(\tube_\alpha)=\bigcap_{\alpha<w} \mathcal
  D_\alpha$,
  cf. also \cite[13.4]{RR}.  The second statement is proven with dual
  arguments.

  (2) The first statement follows from \cite{CB} or \cite[4.5.2]{GT}
  using that $\add\q_w$ is a torsion class in $\mla$. For the second
  statement recall that every module is a direct limit of finitely
  presented modules, so by definition of $\mathcal C_w$ we obtain
  $\subset$. For the reverse inclusion, observe that
  $\p_w\subset\q_w^o=\mathcal C_w$ and $\mathcal C_w$ is closed under
  direct limits.

(3)  follows immediately from the definitions and the separating condition.

(4) For $\alpha\in\Q^+$ we know e.g.~from \cite[\S 14]{RR} that
$\mathcal D_\alpha=(\mathcal C_\alpha)^\perp$, and we infer that
$\bigcap_{\R^+\ni v< w}(\mathcal C_v)^\perp\subset \bigcap_{\Q^+\ni
  \alpha< w}\mathcal D_\alpha=\mathcal B_w$.
For the reverse inclusion, pick $B\in\mathcal B_w$ and
$C\in\mathcal C_v$ where $v\in\R^+$ with $v<w$.  Choose
$\alpha\in\Q^+$ with $v<\alpha<w$. Then $\Ext^1_\Lambda(C,B)=0$
because $C\in\mathcal C_\alpha$ and $B\in\mathcal D_\alpha$. For the
second equality we refer to \cite[13.4]{RR}. Finally, since
$\mathcal C_v\subset \mathcal C_w$ for all $v<w$, we have
$(\mathcal C_w)^\perp\subset\mathcal B_w$.
\end{proof}

We obtain the following description of modules of irrational slope.
\begin{thm}\label{irrational slope} 
Let $w\in\R^+\setminus\Q^+$. 
\begin{enumerate}
\item $\mathbf L_w$ is the only tilting module of slope $w$ up to equivalence.

\item  $\mathbf W_w$ is the only cotilting module of slope $w$ up to equivalence.

\item A module has slope $w$ if and only if it is a pure submodule of
  a product of copies of $\mathbf W_w$, or equivalently, it is a
  pure-epimorphic image of a direct sum of copies of $\mathbf L_w$.
\end{enumerate}
\end{thm}
\begin{proof}
  By construction,
  $\mathbf L_w\in\mathcal B_w\cap{}^\perp(\mathcal B_w)\subset\mathcal
  B_w\cap\mathcal C_w$,
  and
  $\mathbf W_w\in\mathcal C_w\cap(\mathcal C_w)^\perp\subset\mathcal
  C_w\cap\mathcal B_w$ have slope $w$.

  (1) Let $T$ be a tilting module of slope $w$. As in Lemma
  \ref{SandT}, it follows that the corresponding resolving subcategory
  $\mathcal S={}^\perp(T^\perp)\cap\mla=\add\p_w$, and Theorem
  \ref{res} implies that $T$ is equivalent to $\mathbf L_w$.

  (2) Let now $C$ be a cotilting module of slope $w$. Then
  ${}^\perp\B_w\subset {}^\perp C=\Cogen C\subset \C_w$. Since
  $\p_w\subset {}^\perp \B_w$ and ${}^\perp C$ is closed under direct
  limits, we infer $\varinjlim\p_w\subset {}^\perp C$. Lemma
  \ref{BwCw}(2) now yields $\C_w={}^\perp C$, so $C$ is equivalent to
  $\mathbf W_w$.

  (3) First of all, note that $\mathcal M_w$ is definable and
  therefore closed under direct sums, direct products, pure submodules
  and pure-epimorphic images, see for instance \cite[4.2 and
  4.3]{B}. So, we only have to prove the only-if part. Let $M$ be a
  module of slope $w$. As shown in \cite{ATT}, there is a short exact
  sequence $$0\to M\to C_0\to C_1\to 0$$ where $C_0\in(\C_w)^\perp$
  and $C_1\in \C_w$. As $M\in\C_w$, also the middle term
  $C_0\in\C_w$. As $\C_w$ consists of modules of projective dimension
  at most one by \cite[5.4]{RR}, the class $(\C_w)^\perp$ is closed
  under epimorphic images and thus $C_1\in(\C_w)^\perp$. We conclude
  that $C_0,C_1\in\C_w\cap(\C_w)^\perp=\Prod \mathbf W_w$.

  It remains to prove that the sequence is pure-exact, that is, that
  it stays exact under the functor $\Hom_\Lambda(F,-)$ for any
  finitely generated $\Lambda$-module $F$. We can assume w.l.o.g. that
  $F$ is indecomposable with $\Hom_\Lambda(F,C_1)\not=0$, hence
  $F\in\p_w$. Then the claim follows from the fact that
  $\Ext^1_\Lambda(F,M)=0$ as $M\in\B_w=\p_w\,^\perp$.

The second statement is proven dually. 
 \end{proof}

 The tilting module $\mathbf L_w$ can be constructed in a similar way
 as the Lukas tilting module in \cite{KT,Lu}.
\begin{prop}\label{irrationalLukas}
  Let $w\in\R^+\setminus\Q^+$. Given a sequence of rational numbers
  $\alpha_1<\alpha_2<\ldots<w$ converging to $w$, there is a chain of
  modules in $\add\p_w$
$$\Lambda=P_0\subset P_1\subset P_2\subset\ldots$$
with $P_i$ of slope ${\alpha_i}$ for all $i\ge 1$ and
$P_{i+1}/P_i\in\add\p_w$ for all $i\ge 0$, such that the modules
$L_0=\bigcup_{i\in\N} P_i$ and $L_1=L_0/\Lambda$ give rise to a
tilting module $\mathbf L_w=L_0\oplus L_1$ with tilting class
$\Gen\mathbf L_w=\B_w$.
 \end{prop}
 \begin{proof}
   Start with $P_0=\Lambda$. Given $P_i$, which by assumption belongs
   to $\add\p_{\alpha_{i+1}}\subset\C_w$, take a special
   $(\C_w)^\perp$-approximation as in \cite{ATT}
$$0\to P_i\stackrel{f}{\longrightarrow} W_0\to W_1\to 0$$ with
$W_0,W_1\in\Prod W_w$. Observe that
$W_0\in(\C_w)^\perp\subset\Qcal_{\alpha_{i+1}}=\varinjlim
\q_{\alpha_{i+1}}$ by Lemma \ref{BwCw} (2) and (4). So $f$ factors
through a map $P_i\stackrel{f'}{\longrightarrow}Q$ with $Q\in\add
\q_{\alpha_{i+1}}$, and by the separation condition $f'$ factors
through a map $P_i\stackrel{f''}{\longrightarrow}P_{i+1}$ with
$P_{i+1}\in\add \tube_{\alpha_{i+1}}$ of slope $\alpha_{i+1}$. We
obtain a commutative diagram  
   \[\xymatrix{
         0\ar[r] &P_i\ar[d]^{=}\ar[r]^{f''}&P_{i+1}\ar[d]^{h}\ar[r]^{}&Z_{i+1}\ar[d]^{g}\ar[r]^{}&0
 \\
 0\ar[r] &P_i \ar[r]^{f}&W_0\ar[r]^{}&W_1\ar[r]^{}&0
 }\]
 where $\Ker g\cong \Ker h\subset P_{i+1}\in\C_w$ cannot have
 submodules in $\q_w$. Then $Z_{i+1}$ cannot have submodules
 $U\in\q_w$, because $g\mid_{U}:U\to W_1\in\C_w$ would have to be zero
 and thus $U\subset\Ker g$. So, we conclude that
 $P_{i+1}/P_i\cong Z_{i+1}\in\add\p_w$.

 The $\p_w$-filtered modules $L_0=\bigcup_{i\in\N} P_i$ and
 $L_1=L_0/\Lambda$ then belong to
 ${}^\perp(\p_w\,^\perp)={}^\perp\B_w$ by \cite[3.2.4]{GT}.  Further,
 the direct limit $L_0$ of the chain $ P_1\subset P_2\subset\ldots$
 has slope $w$ by \cite[13.4]{RR}. This implies that $L_0$ and $L_1$
 belong to $\B_w$ and therefore to
 $\Add\mathbf L_w=\B_w\cap{}^\perp\B_w$. Now the claim follows easily.
 \end{proof} 

\smallskip

\subsection{\sc Pure-injective modules}\label{piclass}
By Theorem \ref{classpi}, the pure-injective modules of rational slope
$w$ are precisely the modules of the form $M=M'\oplus M''$ where
$M'\in\Prod\tube_w$ and $M''\in\Add\mathbf W_w=\Prod\mathbf W_w$. By
results of Harland and Prest \cite{HP}, pure-injective modules of
irrational slope can be superdecomposable when the ground field $k$ is
countable. Observe, however, that the superdecomposable part does not
play a role when computing the cotilting class $\C_w$, as shown by
Theorem \ref{superdec}. Moreover, the class of pure-injectives is
described as follows.

\begin{cor}\label{pirrat}
  If $w\in\R^+\setminus\Q^+$, then $\mathbf W_w$ is a pure-injective,
  non-$\Sigma$-pure-injective module, and $\Prod\mathbf W_w$ is the
  class of all pure-injective modules of slope $w$.
\end{cor}
 \begin{proof}
   The cotilting module $\mathbf W_w$ is pure-injective by \cite{B1},
   and $\Prod\mathbf W_w$ is the class of all pure-injective module of
   slope $w$ by Theorem \ref{irrational slope}(3).

   Assume that $\mathbf W_w$ is $\Sigma$-pure-injective. Then so is
   every product of copies of $\mathbf W_w$ and any pure submodule of
   such product, yielding $\Mcal_w=\Prod \mathbf W_w$ by Theorem
   \ref{irrational slope}(3). Hence
   $\mathbf L_w\in\Mcal_w\subset(\Ccal_w)^\perp$, and as $\Ccal_w$
   consists of modules of projective dimension at most one,
   $\Gen\mathbf L_w\subset(\Ccal_w)^\perp$, and
 $(\mathcal C_w)^\perp=\mathcal B_w$ by Lemma \ref{BwCw}(4). 
 From $\mathcal Q_w\subset\mathcal B_w$ we deduce that
 $(\mathcal Q_w,\mathcal C_w)$ is a split torsion pair satisfying the
 assumptions of Proposition \ref{sheavesashearts} (for assumption (v)
 observe that $\C_w\subset\C_v$ for some rational $v>w$, and use Lemma
 \ref{5.4}). So, the heart $\mathcal A_w$ of the corresponding
 t-structure in $\Db$ is equivalent to the category category
 $\Qcoh\mathbb{Y}$ of quasi-coherent sheaves over a noncommutative
 curve of genus zero $\mathbb Y$.

 We want to lead this to a contradiction. To this end, we investigate
 the category $\mathcal H_0$ of finite length objects in
 $\mathcal H={\rm fp}(\mathcal A_w)\sim\coh\mathbb{Y}$. We know
 e.g. from \cite[10.1]{L} that there is a family of connected
 uniserial Hom-orthogonal length categories
 $\mathcal U_y,y\in \mathbb Y,$ such that all $\mathcal U_y$ have
 finite $\tau$-period and
 $\mathcal H_0=\bigcup_{y\in\mathbb Y} \mathcal U_y$. So, if $S$ is a
 simple object in $\mathcal H_0$, then its injective envelope $E(S)$
 has only finitely many non-isomorphic composition factors. Note that
 $S$ is of the form $S=Y[1]$ with $Y\in\p_w$ or $S=Q$ with $Q\in\q_w$.
 In the first case, $Y\in\p_\alpha$ for some $\alpha<w$, and there is
 $\alpha<\beta<w$ such that $E(S)$ has all composition factors in
 $\p_\beta[1]$ and therefore $\Hom_\A(\tube_\beta[1],E(S))=0$. On the
 other hand, $Y$ is cogenerated by $\tube_\beta$, so there is a
 non-zero map $Y\to B$ for some indecomposable module
 $B\in\tube_\beta$, yielding a monomorphism $S\to B[1]$ and thus a
 non-zero map $B[1]\to E(S)$ in $\A_w$, a contradiction.  This shows
 that the simple objects in $\mathcal H_0$ are all of the form $S=Q$
 with $Q\in\q_w$, so they belong to the torsion-free class of the
 torsion pair $(\mathcal C_w[1],\mathcal Q_w)$ in $\A_w$. But then the
 noetherian tilting object $V=\Lambda[1]\in\mathcal C_w[1]$ cannot
 have a simple quotient, again a contradiction.
\end{proof}

In order to determine the pure-injectives of slope $0$ or $\infty$, we
need to investigate the non-stable tubular families $\tube_0$ and
$\tube_\infty$.  It will be convenient to work with sheaves rather
than with modules.  We fix a canonical trisection $(\p,\tube,\q)$ of
$\mla$. The corresponding torsion pair $(\Gen\q,\C)$ induces a
t-structure whose heart $\A$ can be identified with the category
$\Qcoh\XX$ of quasi-coherent sheaves over a noncommutative curve
$\mathbb X$ of genus zero and tubular type, cf.~Section \ref{setup}.

We know from Section \ref{hearts of cotilting} that $\Qcoh\XX$ admits
a coherent tilting sheaf $V$ with endomorphism ring $\Lambda$ yielding
an equivalence $H_V=\Hom_\A (V,-):\Gen V\to \C$.  We now proceed as in
\cite[Section 4.9]{DirkDiss}, where more details can be found. Let
$V_1,\ldots,V_m$ be the indecomposable direct summands of $V$ having
maximal slope $\alpha$ as sheaves in $\Qcoh\XX$. Under the functor
$H_V$, they correspond to the indecomposable projective modules
contained in the non-stable family $\tube_0$. We write
$V=V_0\oplus V_{max}$ where $V_{max}=\bigoplus_{i=1}^m V_i$, and we
denote by $\Lambda_0=\Lambda/\Lambda e\Lambda \cong \End_\A V_0$ the
algebra induced by the idempotent $e\in\Lambda$ corresponding to the
direct summand $V_{max}$.  The sheaves $V_1,\ldots,V_m$ are arranged
in a union $\mathfrak W$ of disjoint wings inside the stable tubular
family $\mathfrak t_\alpha$ in $\Qcoh\XX$ consisting of all coherent
sheaves of slope $\alpha$. There are precisely $m$ rays starting (and
$m$ corays ending) in $\mathfrak W$.
 
Now a sheaf $X\in\Qcoh\XX$ of slope $\alpha$ corresponds to a
$\Lambda$-module in $\C$ under the functor $H_V$ if and only if
$\Ext_\A^1(V,X)=0$. By Serre duality and slope arguments this amounts
to $\Hom_\A(X,\tau V)=\Hom_\A(X,\tau V_{max})=0$, that is,
$\Ext_\A^1(V_{max},X)=0$.  If additionally $\Hom_\A(V_{max},X)=0$,
then $H_V(X)=\Hom_\A(V_0,X)$ is a $\Lambda_0$-module.
Notice that the algebra $\Lambda_0$ is tame concealed, its
preprojective component agrees with $\p_0$, and its (stable) tubular
family is obtained from $\mathfrak t_\alpha$ by removing the $m$ rays
starting in $\mathfrak W$ and the $m$ corays ending in
$\tau\mathfrak W$. Moreover, there is a homological ring epimorphism
$\lambda:\Lambda\to\Lambda_0$.


Let us turn to the Pr\"ufer sheaves of slope $\alpha$. We claim that
they correspond to indecomposable pure-injective $\Lambda$-modules of
slope $0$. For a proof, we switch to the category $\Qcoh\XX_\alpha$
whose indecomposable finite length objects are given by the tubular
family $\mathfrak t_\alpha$ (here $\mathbb X_\alpha$ is a
noncommutative curve of tubular type, which is isomorphic to
$\mathbb X$ in case $k$ is algebraically closed), and we apply Lemma
\ref{JL}. More precisely, a Pr\"ufer sheaf $X$ of slope $\alpha$ is
injective when viewed inside $\Qcoh\XX_\alpha$, cf.~Section
\ref{setup} or \cite[Prop.~3.6]{AK}.  So $X$ is a pure-injective sheaf
in $\Qcoh\XX$. Further, $\Ext_\A^1(V,X)=0$, and the functor $H_V$,
which preserves direct limits and direct products, maps $X$ to a
pure-injective $\Lambda$-module. This module has slope $0$ because the
sheaves in $\mathfrak t_\alpha$ correspond to $\Lambda$-modules in
$\tube_0$ except for the $m$ corays ending in $\tau\mathfrak W$.  We
denote by $X_1,\ldots,X_m$ the $\Lambda$-modules which correspond to
Pr\"ufer sheaves originating in the $m$ rays from $\mathfrak W$; the
remaining Pr\"ufer sheaves correspond to the Pr\"ufer
$\Lambda_0$-modules.

Similarly, the adic sheaves of slope $\alpha$ correspond to
$\Lambda$-modules if and only if they don't arise from corays ending
in $\tau\mathfrak W$, in which case they are indecomposable
pure-injective $\Lambda$-modules and agree with the adic modules over
$\Lambda_0$. Moreover, they have slope $0$ because $\mathcal M_0$ is
closed under inverse limits by \cite[Lemma 9.10]{BH2}.

The case $w=\infty$ is obtained by duality, since the role of $0$ and
$\infty$ is swapped when turning to left modules via the duality
$D=\Hom_k(-,k)$.
So, there is a tame concealed factor algebra $\Lambda_\infty$ of
$\Lambda$ whose preinjective component agrees with $\q_\infty$, and
there 
are indecomposable pure-injective $\Lambda$-modules of slope $\infty$
which, except from, say, $\ell$ modules denoted by
$Y_1,\ldots,Y_\ell$, are precisely the adic
$\Lambda_\infty$-modules. Furthermore, also the Pr\"ufer modules over
$\Lambda_\infty$ are pure-injective $\Lambda$-modules.

Finally, for both $w=0$ and $w=\infty$, the generic module over
$\Lambda_w$ occurs as a direct summand in a direct product of copies
of any Pr\"ufer $\Lambda_w$-module, and so it is a pure-injective
$\Lambda$-module of slope $w$.

\smallskip

Let us now summarize our findings.
 
\begin{thm}\label{Ziegler} 
 The following is a complete list of the indecomposable pure-injective
$\Lambda$-modules:
\begin{enumerate}
 \item the finite dimensional indecomposable modules,
 \item the Pr\"ufer modules, the adic modules, and the generic module
   of slope $w$ with $w\in\Q^+$,
 \item the indecomposable modules in $\Prod\mathbf W_w$ with
   $w\in\R^+\setminus \Q^+$,
 \item the Pr\"ufer modules, the adic modules, and the generic module
   over $\Lambda_0$ and $\Lambda_\infty$,
 \item the modules $X_1,\dots, X_m$ and $Y_1,\ldots,Y_\ell$ defined
   above.
 \end{enumerate}
\end{thm}
\begin{proof}
  By the discussion above, all modules in the list are indecomposable
  pure-injective, so we only have to show that the list is
  complete. Every infinite dimensional indecomposable pure-injective
  module has a slope $w$ by \cite[13.1]{RR}.
  Combining 
  \cite[Lemma 50]{H} with Corollary \ref{pirrat}, we get the statement
  for $w\in\R^+$.

  Let now $w\in\{0,\infty\}$.  We discuss the case $w=0$, the case
  $w=\infty$ is obtained by duality.  By Proposition \ref{pi} we can
  assume that our module is not a $\Lambda_0$-module. Keeping the
  notation as above, it then has the form $H_V(I)$ for an
  indecomposable sheaf $I\in\Qcoh\XX$ with
  $\Hom_\A(\mathfrak W,{I})\not=0$.  Consider the canonical exact
  sequence $$0\to t(I)\to I\to I/t(I)\to 0$$ induced by the torsion
  pair $(\Gen{\mathfrak t}_\alpha, \mathfrak F_\alpha)$ generated by
  $\mathfrak t_\alpha$ in $\Qcoh\XX$. Observe first that $t(I)$ cannot
  vanish because $I$ is not torsionfree by our assumption. By
  \cite[Prop.~3.7]{AK}, the sheaf $t(I)$ then has an indecomposable
  pure-injective summand which is either coherent or a Pr\"ufer
  sheaf. Since the sequence is pure-exact by \cite[Rem.~3.3]{AK}, this
  summand must coincide with the indecomposable sheaf $I$, and our
  module $H_V(I)$ must then be isomorphic to one of $X_1,\ldots,X_m$.
\end{proof}

\begin{rem} (1) Every module in $\Ccal_w,\,w\in\R^+\cup\{0\}$, has
  projective dimension at most one, and every module in
  $\Bcal_w,\,w\in\R^+\cup\{\infty\}$, has injective dimension at most
  one, as a consequence of Lemma \ref{5.4}.  In particular, all
  modules in $\p_\infty$ have projective dimension at most one, and
  all modules in $\q_0$ have injective dimension at most one.

\smallskip

(2) Let $w\in\{0,\infty\}$. There are a cotilting module $\mathbf W_w$
and a tilting module $\mathbf L_w$ of slope $w$ such that
$\Ccal_w={}^\perp\q_w=\Cogen\mathbf W_w$ and
$\Bcal_w=\p_w\,^\perp=\Gen\mathbf L_w$.  This is a consequence of
Theorem \ref{res} and (1).  That $\mathbf W_w$ and $\mathbf L_w$ have
slope $w$ is shown as in the proof of Theorem \ref{irrational slope}.

\smallskip

(3) Up to equivalence,
$\mathbf W_0= \lambda_\ast(W) \oplus X_1\oplus\ldots\oplus X_m$, where
$W $ is the direct sum of the generic and all Pr\"ufer modules over
$\Lambda_0$ and $\lambda_\ast:\Mla_0\to\Mla$ is the embedding given by
the restriction of scalars along the ring epimorphism
$\lambda:\Lambda\to\Lambda_0$.

Indeed, recall first that $\mathbf W_0=H_V(I)$ where $I$ is the sum of
the generic and the Pr\"ufer sheaves of slope $\alpha$ in $\Qcoh\XX$
(with $\alpha$ as above). We switch again to $\Qcoh\XX_\alpha$, which
can be viewed as the heart of the faithful torsion pair in $\Qcoh\XX$
generated by the class $\mathfrak q_\alpha$ of coherent sheaves of
slope $>\alpha$. Now $I$ corresponds to an injective cogenerator of
$\Qcoh\XX_\alpha$, so arguing as in Subsection \ref{tilting objects},
we see that $I$ satisfies conditions (C1)-(C3) and
$\Cogen \,I={}^\perp I=\Ker\Hom_\A(\mathfrak q_\alpha,-)$.

Next, we turn to our fixed canonical trisection $(\p,\tube,\q)$ of
$\mla$ with the induced split torsion pair $(\Gen\q,\C)$ in $\Mla$ and
tilted torsion pair $(\Gen V, \Ker\Hom_\A(V,-))$ in $\A=\Qcoh\XX$. The
torsion class $\Gen V= V^\perp$ contains $\mathfrak q_\alpha$ by slope
arguments. This implies that $\Ker\Hom_\A(\mathfrak q_\alpha,-)$
contains the corresponding torsionfree class $\Ker\Hom_\A(V,-)$, and
every module $Q\in\Gen\q$ satisfies $\Ext^1_\A(Q,I)=0$.

Now we are ready to prove our claim.  As in Lemma \ref{tilting in
  hearts}, we see that $\mathbf W_0$ verifies conditions (C2) and
(C3). Of course, $\Cogen\mathbf W_0\subset\mathcal C_0$. Further,
$\mathcal C_0\subset {}^\perp \mathbf W_0$, because any module
$C\in\C_0$ belongs to $\C$ and corresponds to a sheaf in
$\Ker\Hom_\A(\mathfrak q_\alpha,-)$ and thus satisfies
$\Ext^1_\Lambda(C,\mathbf W_0)\cong\Ext^1_\A(C[1],I)=0$ by Lemma
\ref{ExtHom}(d).  Let us check
$ {}^\perp \mathbf W_0\subset\Cogen\mathbf W_0$. Take
$M\in {}^\perp \mathbf W_0$, denote by $M'={\rm Rej}_{\mathbf W_0}(M)$
the intersection of all kernels of homomorphisms $M\to\mathbf W_0$,
and set $\overline{M}=M/M'$. Then $\overline{M}\in \Cogen \mathbf W_0$
belongs to ${}^\perp \mathbf W_0$ and to $\C_0$, and in particular it
has projective dimension at most one by (1). So
$\Ext^1_\Lambda(\overline{M}, \mathbf
W_0)=\Ext^2_\Lambda(\overline{M}, \mathbf W_0)=0$.
Applying the functor $\Hom_\Lambda(-, \mathbf W_0)$ to the exact
sequence $0\to M'\to M\to \overline{M}\to 0$ thus yields a long exact
sequence
$0\to \Hom_\Lambda(\overline{M}, \mathbf W_0)\to\Hom_\Lambda(M,
\mathbf W_0)\to\Hom_\Lambda(M', \mathbf
W_0)\to\Ext^1_\Lambda(\overline{M}, \mathbf W_0)\to\Ext^1_\Lambda(M,
\mathbf W_0)\to \Ext^1_\Lambda(M', \mathbf W_0)\to 0$
where the first map is an isomorphism by construction, and the fourth
and fifth term vanish. Then
$\Hom_\Lambda(M', \mathbf W_0)=\Ext^1_\Lambda(M', \mathbf W_0)=0$,
which implies $M'=0$ by condition (C3), and proves that
$M\in\Cogen\mathbf W_0$.  We conclude that $\mathbf W_0$ is a
cotilting $\La$-module cogenerating $\C_0$.
\end{rem}

 \subsection{\sc Tilting modules and sheaves}\label{allhaveslope}
 We now turn to the classification of tilting modules.  Let us
 summarize our findings in Sections \ref{rational} and
 \ref{irrational}.
\begin{cor}\label{resume}
  The tilting and cotilting modules of slope $w\in\Q^+$ are
  parametrized by the pairs $(Y,P)$ given by a branch object $Y$ in
  $\tube_w$ and a subset $P\subset\mathbb X_w$, where $\mathbb X_w$ is
  a noncommutative curve of genus zero (again tubular and derived
  equivalent to $\mathbb X$ by \cite[8.1.6]{Ku}) parametrizing the
  family $\tube_w$. The tilting and cotilting modules of irrational
  slope $w$ are equivalent to $\mathbf L_w$ and $\mathbf W_w$,
  respectively.
\end{cor}

We are going to see that all tilting modules in the ``central part''
of $\Mla$, that is, contained in a smallest $\C_w$ with
${0<w<\infty}$, arise in this way. This will enable us to recover the
classification of the tilting sheaves over a noncommutative curve of
genus zero of tubular type from \cite{AK}. We first need a preliminary
result.
\begin{lemma} \label{1} Let $T$ be a large tilting module with
  $\Scal={}^\perp(T^\perp)\cap\mla$.
\begin{enumerate}
\item Let $w\in\R^+\cup\{0,\infty\}$. Then $T\in\Ccal_w$ if and only
  if $\Scal\subset\Ccal_w$.
\item If $T\in\mathcal C_w$ with $w\in\Q^+$, and
  $\Scal\cap\tube_w\not=\emptyset$, then $T$ has slope $w$.
\end{enumerate}
\end{lemma}
\begin{proof}
(1) is shown as in Lemma \ref{SandT}(1). 

(2) If $\Scal$ contains a ray from $\tube_w$, then we know from Lemma
\ref{SandT} that $\Add T$ contains a Pr\"ufer module $S_\infty$ of
slope $w$.  Then $\Ext^1_\Lambda(S_\infty,T)
=0$, and using that $S_\infty$ has projective 
dimension at most one, it follows from Lemma \ref{0} that $T$ cannot
have non-zero factor modules in $\Pcal_w$.
Considering the torsion pair $(\Bcal_w,\Pcal_w)$, 
we infer that $T\in\Bcal_w$, so $T$ has slope $w$.

If $\Scal$ does not contain a complete ray from $\tube_w$, then we
know from Lemma \ref{SandT} and Proposition \ref{shapeofT} that $T$
has the form $T=Y\oplus M$ where $0\not=Y\in\add\tube_w$ and
$M\in\Fcal_w$.  Let $\Ucal$ be the set of regular composition factors
of $Y$. Then $\Lambda_\Ucal$ is a concealed canonical algebra of
domestic type (cf.~\cite[Sec.~2]{AK}), and by Proposition
\ref{downtoLukas}, $M$ is a large torsion-free tilting module over
$\Lambda_\Ucal$.  By the classification in Theorem \ref{classdomestic}
it follows that $M$ is equivalent to the Lukas tilting module over
$\Lambda_\Ucal$. In particular $\Hom_{\Lambda_\Ucal}(M,P')=0$ for all
$P'\in\p_\Ucal$. Now every $P\in\p_w$ embeds in
$P\otimes_\Lambda\Lambda_\Ucal\in\p_\Ucal$ by Proposition
\ref{prop:sumofprufer}, thus also $\Hom_{\Lambda}(M,P)=0$.  We
conclude that $M\in\Bcal_w$, and $T$ has slope $w$.
\end{proof}

\begin{thm}\label{limits}
  Let $T$ be a large tilting module, and assume there is $w\in\R^+$
  such that $T\in\Ccal_w$ but $T\not\in\Ccal_\alpha$ for any
  $\alpha<w$. Then $T$ has slope $w$.
\end{thm}
\begin{proof}
  First of all, we know from Lemma \ref{1} that
  $\Scal={}^\perp(T^\perp)\cap\mla\subset\Ccal_w$, and we can assume
  w.l.o.g. $\Scal\subset\add\p_w$. This is clear if $w$ is irrational,
  because $\Ccal_w\cap\mla=\add\p_w$, and for $w\in\Q^+\cup\{\infty\}$
  it follows from Lemma \ref{1}(2).

  Furthermore, the assumption on $w$ implies that $\Scal$ is not
  contained in $\add\p_\alpha$ for any $\alpha<w$. So there is an
  increasing sequence of rational numbers $\alpha_1<\alpha_2<\ldots<w$
  converging to $w$ such that
  $\Scal\cap\tube_{\alpha_i}\not=\emptyset$ for all $i$. Let us
  consider the increasing sequence of
  subcategories $$\Scal_1\subset\Scal_2\ldots\subset\mla$$ given by
  $\Scal_i=\Scal\,\cap\,\Ccal_{\alpha_i}$ for all $i\in\N$. Since
  $\Scal$ and $\Ccal_{\alpha_i}={}^\perp\q_{\alpha_i}$ are resolving
  subcategories, the $\Scal_i$ are resolving subcategories of $\mla$
  giving rise to a decreasing sequence of tilting
  classes $$\Scal_1\,^\perp\supset\Scal_2\,^\perp\supset\ldots$$ For
  each $i\in\N$ let $T_i$ be a tilting module with
  $T_i\,^\perp=\Scal_i\,^\perp$. By the bijection in Theorem \ref{res}
  we have $\Scal_i={}^\perp(T_i^\perp)\cap\mla$, and since
  $\Scal_i\cap\tube_{\alpha_i}\not=\emptyset$, we infer from Lemma
  \ref{1} that $T_i$ has slope $\alpha_i$.
 
  Next we observe that $\Scal=\bigcup_{i\in\N}\Scal_i$. Indeed, if
  $X\in\Scal$, and $X$ is indecomposable w.l.o.g., then $X\in\p_w$
  either belongs to $\p_0$ and is therefore contained in all
  $\Scal_i$, or it belongs to $\tube_\alpha$ for some $\alpha<w$ and
  is therefore contained in $\Scal_i=\Scal\cap\Ccal_{\alpha_i}$ for
  $i$ with $\alpha<\alpha_i<w$.
 
  It follows that
  $\Scal^\perp=\bigcap_{i\in\N}\Scal_i\,^\perp=\bigcap_{i\in\N}T_i\,^\perp$
  and thus $T\in\bigcap_{i\in\N}T_i\,^\perp$.  We claim that
  $T\in\Bcal_w$, which will yield that $T$ has slope $w$. By Lemma
  \ref{BwCw}, the claim amounts to showing that $T\in\Bcal_v$ for all
  $v<w$.  So take $v<w$ and $i\in\N$ such that $v<\alpha_i
  <w$.
  Observe that $T_i$ belongs to $\Bcal_{\alpha_i}\subset\Bcal_v$
  because it has slope $\alpha_i$. As $\Bcal_v$ is a torsion class,
  also $T_i\,^\perp=\Gen T_i\subset \Bcal_v$. But then also
  $T\in T_i\,^\perp$ belongs to $\Bcal_v$, which completes the proof.
\end{proof}

\begin{rem} By \cite{BC} one can realize a tilting module $T$ as above
  as a direct limit of tilting modules of increasing slope
  $\alpha_1<\alpha_2<\ldots<w$.

\smallskip

Furthermore, Theorem \ref{limits} has a dual version: if $C$ is a
large cotilting module, and there is $w\in\R^+$ such that
$C\in\Bcal_w$ but $C\not\in\Bcal_\alpha$ for any $\alpha>w$, then $C$
has slope $w$. This is proved by using that the dual module $D(C)$ is
a tilting module in the central part of $\LMla$.\end{rem}

\smallskip

Let now $\mathbb X$ be a noncommutative curve of genus zero of tubular
type. The slope of an indecomposable coherent sheaf is defined as the
ratio of the degree by the rank. It is a rational number, unless the
sheaf has finite length, in which case the rank is zero and the slope
$\infty$. The coherent sheaves of a given slope $w\in\Q\cup\{\infty\}$
form a tubular family denoted by $\widehat{\tube}_w$. One then extends
the notion of slope to all quasi-coherent sheaves like in Section
\ref{slope}. For details, we refer to \cite{AK}. We can now recover
the classification of large quasi-coherent tilting sheaves over
$\mathbb X$ from \cite[Thm.~7.14]{AK}.
\begin{cor}\label{classofsheaves}
  Let $\mathbb X$ be a noncommutative curve of genus zero of tubular
  type. Then every large tilting sheaf in $\Qcoh\XX$ has a slope
  $w\in\R\cup\{\infty\}$. The large tilting sheaves of slope
  $w\in\Q\cup\{\infty\}$ are parametrized by the pairs $(Y,P)$ given
  by a branch object $Y$ in $\widehat{\tube}_w$ and a subset
  $P\subset\mathbb X_w$, where $\mathbb X_w$ is a noncommutative curve
  of genus zero parametrizing the family $\widehat{\tube}_w$. The
  tilting sheaves of irrational slope $w$ are equivalent to the Lukas
  tilting sheaf $\widehat{\mathbf L_w}$.
\end{cor}
\begin{proof} For every tilting object $\widehat{T}$ in $\A=\Qcoh\XX$
  one can find a tilting bundle $T_{cc}$ in $\coh\XX$ such that
  $\widehat{T}\in\Gen T_{cc}$, cf.~\cite[Lem.~7.10]{AK}.  Then
  $\Lambda=\End_\A T_{cc}$ is a concealed canonical tubular algebra
  derived equivalent to $\A$, and $\Mla$ can be viewed as the heart of
  the torsion pair $(\Gen T_{cc},\Ker \Hom_\A(T_{cc},-))$ in
  $\A$. Under the tilting functor $\Hom_\A(T_{cc},-)$, the tubular
  family in $\coh\XX$ formed by the indecomposable sheaves of finite
  length becomes a tubular family $\tube_\alpha$ in the AR-quiver of
  $\Lambda$. Since $\coh\XX$ has neither projective nor injective
  objects, $\tube_\alpha$ is {stable}, that is, $\alpha\in\Q^+$. The
  torsion pair $(\Gen T_{cc},\Ker \Hom_\A(T_{cc},-))$ can now be
  viewed as the tilt of the torsion pair $(\Qcal_\alpha,\C_\alpha)$ in
  $\Mla$, that is, $\Gen T_{cc}=\C_\alpha[1]$ and
  $\Ker \Hom_\A(T_{cc},-)=\Qcal$, compare Sections \ref{hearts of
    cotilting} and \ref{setup}. We infer from Corollary \ref{shape2}
  that $\widehat{T}=T[1]$ for a tilting $\Lambda$-module
  $T\in\C_\alpha$, and $T$ must have slope $w\le \alpha$ by Theorem
  \ref{limits}. The statement now follows from Corollary
  \ref{resume}. Here $\widehat{\mathbf L_w}=\mathbf L_w[1]$.
\end{proof}


\bigskip

\bigskip
 

\end{document}